\documentclass[12pt]{article}

\usepackage{fancyhdr, amsmath,amsfonts, url, subfigure, dsfont, mathrsfs, graphicx, epsfig, subfigure, amsthm, tikz,xcolor}
\usepackage{comment}

    \usepackage{epsfig}
    \usepackage{graphicx}
    \usepackage{color}

\newlength{\guillotine}
\newtheorem{thm}{Theorem}[section]

\newtheorem{lemma}[thm]{Lemma}

\newtheorem{prop}[thm]{Proposition}

\newtheorem{definition}[thm]{Definition}
\newtheorem{example}[thm]{Example}

\theoremstyle{remark}
\newtheorem{rem}[thm]{Remark}
\theoremstyle{plain}
\newcommand{\thistheoremname}{}
\newtheorem{genericthm}[thm]{\thistheoremname}

\newtheorem*{genericthm*}{\thistheoremname}
\newenvironment{namedthm*}[1]
  {\renewcommand{\thistheoremname}{#1}%
   \begin{genericthm*}}
  {\end{genericthm*}}

\begin{document}

\title{The growth and distribution of large circles  on 
 translation surfaces}

\author{P. Colognese and M. Pollicott\thanks{Supported by the ERC Grant 833802-Resonances and  EPSRC grant EP/T001674/1.
The second author is  very grateful to A. Math\'e for  earlier discussions on an earlier simplified model.  
}}
\newcommand{\Addresses}{{
  \bigskip
  \footnotesize
  P.~Colognese, \textsc{Department of Mathematics, Warwick University,
    Coventry, CV4 7AL-UK}\par\nopagebreak
  \textit{E-mail address} \texttt{p.colognese@warwick.ac.uk}\\
  
  M.~Pollicott, \textsc{Department of Mathematics, Warwick University,
    Coventry, CV4 7AL-UK}\par\nopagebreak
  \textit{E-mail address} \texttt{masdbl@warwick.ac.uk}
}}

\date{
   Warwick University\\[2ex]
    \today
}

\maketitle

\abstract{
We consider circles on a translation surface $X$,  consisting of points joined to a common center point by a geodesic of length $R$. We show that as $R \to \infty$ these circles distribute to a measure on $X$ which is equivalent to the area.  In the last section we consider analogous results for closed geodesics.}

\section{Introduction}
Given a closed  Riemannian  surface of constant curvature and genus $g \geq 2$, several authors have considered 
 circles consisting of those points which 
are joined by geodesics of a common length (which can be thought of as radii) to 
a common point (which can be viewed as  the center).   Equivalently, we can consider the projections 
of circles from its  universal cover.
 As the radius tends to infinity, these circles become equidistributed with respect to Haar measure.  In the case $g=1$ this is an easy exercise 
and for $g \geq 2$ this was shown by Randol \cite{randol}.

 We want to consider the  natural generalization to  translation surfaces $X$.
  More precisely, given a point $x \in  X$ and 
$R > 0$,   we can  naturally associate a one dimensional curve 
  $\mathcal C(x, R) \subset X$ consisting of those points
  on $X$ joined
to $x$  
   by a geodesic of length $R$.  
On $X$ the radial geodesics 
are  either 
straight line segments or 
concatenations of 
   straight line segments and saddle connections joining  singularities.   
  Let $\ell(\mathcal C(x, R))$ denote the total length 
  of the one dimensional curve $\mathcal C(x, R)$.
  Despite  the surface being flat (except at the finite set of singularities $\Sigma$), the length of $\mathcal C(x, R)$ actually grows exponentially in $R$ because of how geodesics behave at the singularities.
  In particular, we have that
  $$
h(X) := \lim_{R \to \infty} \frac{1}{R} \log
\ell(\mathcal C(x, R)) \eqno(1.1)
$$
exists and is positive.
We will call $h(X)$ the {\it entropy} of the surface.  
In his PhD thesis,
Dankwart \cite{dankwart} originally  defined  (volume) entropy in terms of \textcolor{black}{ orbital counting.
 In \cite{cp}, we defined a notion of entropy in terms of the rate of growth of the volume of a ball in the universal cover, by analogy with the definition due to Manning for Riemannian manifolds \cite{manning}.}  
More precisely, let $\widetilde X$ be the universal cover for $X$ and let $\mathcal B(\widetilde x, R)$
be a ball in $\widetilde X$ of radius $R>0$ centered at $\widetilde x$.  If we then write $\hbox{\rm Vol}_{\widetilde X}(\mathcal  B(\widetilde x, R))$ for the volume of the ball, then the original definition of (volume) entropy  is
$$
h(X) = \lim_{R \to \infty}  \frac{1}{R} \log \hbox{\rm Vol}_{\widetilde X}( B(\widetilde x, R)). \eqno(1.2)
$$
However, these different definitions are easily seen to be equivalent.  The definition (1.1) in terms of the length  
$\ell(\mathcal C(x, R)) $
 has the slight conceptual  advantage 
that it does not necessitate going to the universal cover since the length of the curve $\mathcal C(x, R)$
 has a natural interpretation on $X$.

Our first main result improves on (1.1)  by giving  the natural asymptotic formula for the length of the curve.

\begin{namedthm*}{Theorem A}[Asymptotic length formula]\label{asymptotic}
There exists $C = C(x) > 0$ such that 
$$
	\ell(\mathcal C(x, R)) \sim C e^{h(X) R} \hbox{ as } R \to \infty  \quad 
	\left(\hbox{i.e.,} \lim_{R\to \infty} \frac{\ell(\mathcal C(x, R))}{ C e^{h(X) R}} = 1\right).
	$$
\end{namedthm*}

The asymptotic formula in Theorem A  is reminiscent of the simpler corresponding result
 for Riemannian surfaces of constant negative curvature 
(without singularities) of genus $g \geq 2$ which is easily deduced from  \cite{huber}. See also work by Margulis \cite{margulis}.

Next we want to give a distributional result for the circles $\mathcal C(x, R)$.
We  define   a family of natural probability measures $\mu_R$ supported on the  sets $\mathcal C(x, R)$ for $R > 0$.  These correspond to the normalized arc length measure on  the curve $\mathcal C(x, R)$.

\begin{definition}
 We can define a family of probability measures $\mu_R$ ($R>0$) on $X$ by
$$
\mu_R(A)=\frac{\ell(\mathcal C(x,R)\cap A)}{\ell(\mathcal C(x,R))},\hbox{ for Borel sets $A\subset X$,}
$$

where $\ell(\mathcal C(x,R)\cap A)$ denotes the one dimensional measure of $\mathcal C(x,R) \cap A$. 
  \end{definition}
  
The next result describes  the convergence of the probability measures $\mu_R$ as the radius $R$ tends to infinity.

 \begin{namedthm*}{Theorem B}[Circle distribution result]\label{distribution}
 The sequence of measures $(\mu_R)_{R>0}$ converge in the weak-star topology to a measure $\mu$ which is equivalent to the volume measure 
 $\hbox{\rm Vol}_X$
 on $X$, i.e.
 $
 \lim_{R \to \infty}
 \int f d\mu_R =  \int f d\mu
  $
  for any $f \in C(X)$.
 \end{namedthm*}
 
 We note that this is not quite a traditional equidistribution result in the sense that  although 
 $\mu$  is  equivalent to the volume measure $\hbox{\rm Vol}_X$, 
 the Radon-Nikodym derivative $d\mu/d\hbox{\rm Vol}_X$ is not constant.

Our method of proof for both theorems is based on an approach using complex functions and Tauberian theorems 
developed in \cite{cp}.

 In \S 2 we describe the basic definitions and examples.   In \S 3 we present the basic approach to describing the 
 growth of circles (and the proof of Theorem A).
  In \S4 we use a similar method to prove a growth rate for weighted balls on the universal cover of $X$ and in \S 5   we use the results from \S 4 to deduce Theorem B. 
  Finally in \S 6  we discuss analogous results to Theorem A and B for closed geodesics.

 \section{Translation surfaces and geodesics}
 In this section, we recall a convenient  definition  of   translation surfaces and their basic properties.  A good reference for this material (and more background) are the surveys \cite{zorich} and \cite{wright}.   
We will use the same notation as in \cite{cp}.
  
  Roughly speaking, a translation surface $X$ is a closed surface endowed with a flat metric except at, possibly, a finite number of singular points such that there is a well defined notion of north at every non-singular point.

Singularities on translation surfaces are cone-points. To see what this means, consider the following construction: let $k\in\mathbb N$ and take $(k+1)$ copies of the upper half plane with the usual metric and $(k+1)$ copies of the lower half plane. Then glue them together along the half infinite rays $[0,\infty)$ and $(-\infty,0]$ in cyclic order (Figure \ref{fig:nbhdsing}). 

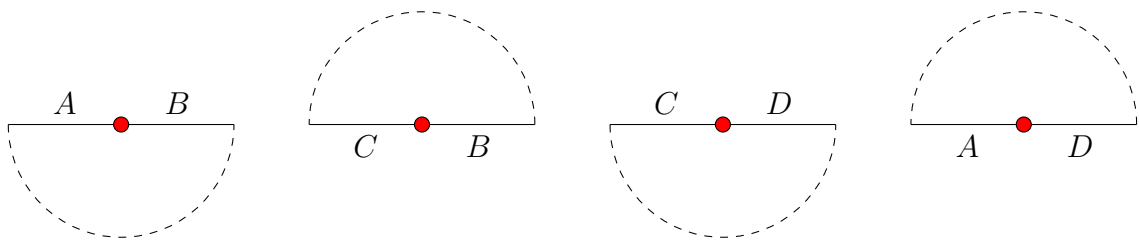
\begin{figure}[h!]
    \centering
    \begin{tikzpicture}
      \draw (0,0)--(3,0);
      \draw (4,0)--(7,0);
      \draw (8,0)--(11,0);
      \draw (12,0)--(15,0);
      
      \draw[dashed] (0,0) arc (180:360:1.5);
        \draw[dashed] (7,0) arc (0:180:1.5);
              \draw[dashed] (8,0) arc (180:360:1.5);
                    \draw[dashed] (15,0) arc (0:180:1.5);
                    
      \draw [fill= red] (1.5,0) circle (0.1);
            \draw [fill= red] (5.5,0) circle (0.1);
                  \draw [fill= red] (9.5,0) circle (0.1);
                        \draw [fill= red] (13.5,0) circle (0.1);

          \node[above] at (0.75,0) {$A$};
                    \node[above] at (2.25,0) {$B$};
                \node[above] at (8.75,0) {$C$};
                      \node[above] at (10.25,0) {$D$};
                  \node[below] at (4.75,0) {$C$};
              \node[below] at (6.25,0) {$B$};
              \node[below] at (12.75,0) {$A$};
             \node[below] at (14.25,0) {$D$};

    \end{tikzpicture}
    \caption{Four half-disks of radius cyclic fashion. A cone point of angle $4\pi$ on a translation surface has a neighbourhood isometric to a neighbourhood of the the origin in the picture. }
    \label{fig:nbhdsing}
\end{figure}

There are a few equivalent definitions which appear in the literature; however, we will present the one which is most suited to our needs.

\begin{definition}
A translation surface is a closed topological surface $X$, together with a finite set of points $\Sigma$ and an atlas of charts to $\mathbb C$ on $X\backslash \Sigma$, whose transition maps are translations. Furthermore, we require that for each point $x\in \Sigma$, there exists some $k\in\mathbb N$ and a homeomorphism of a neighborhood of $x$ to a neighbourhood of the origin in the $2k+2$ half plane construction that is an isometry away from $x$. 
\end{definition}

It is easy to see that the above definition gives a locally Euclidean metric on $X\backslash \Sigma$. The set $\Sigma=\{x_1,\ldots,x_n\}$ is the set of singularities or cone-points on the surface, where the singularity $x_i$ has a cone-angle of the form $2\pi(k(x_i)+1)$ with $k(x_i)\in\mathbb N$.

 In the absence of singularities, the surface is a torus, but  if $X$ has genus at least $2$ then by a simple consideration of the Gauss-Bonnet theorem we see that there must be at least one singularity.
 Henceforth, we will consider the case $\Sigma \neq \emptyset$.

We recall a simple construction for translation surfaces which is particularly useful in giving  examples. 
Let $\mathbb P$ denote a polygon in  the Euclidean plane $\mathbb R^2$ for which every side has an opposite side which is parallel and of the same length.  By identifying these opposite sides we obtain a translation surface.  The vertices may contribute to the singularities and the total angle (i.e., the cone-angle) around any singularity is $2\pi k$, where $k > 1$.


A path which does not pass through singularities in its interior is a locally distance minimizing geodesic if it is a straight line segment. This includes geodesics which start and end at singularities, which are known as {\it saddle connections}. In particular, we will consider  oriented saddle connections.
Let 
$${\mathcal{S}}=\{s_1,s_2, s_3, ...\}$$
 denote the countably infinite family  of oriented saddle connections on the translation surface
ordered by non-decreasing length.   We let $i(s), t(s) \in \Sigma$
denote the initial and terminal singularities, respectively, of the 
oriented saddle connection $s \in \mathcal S$.

A key difference to the Euclidean case is that geodesics 
(i.e., local distance minimizing curves)
can change direction if they go through a singular point. A  pair of line segments ending and beginning, respectively,  at a singular point form a geodesic if the smallest angle between them is at least $\pi$.
In particular, this leaves $2\pi k(x_i)$ worth of angles for the geodesic to emerge from the singularity $x_i$.
This happens because  we are  studying the growth of locally distance minimizing geodesics. 

The following particular type of geodesic will play a key role in our later analysis.

\begin{definition}
We can denote by $p = (s_1, \ldots, s_n)$ an (allowed)  finite word of oriented saddle connections corresponding to a geodesic path, which we call saddle connection paths.
We denote by $|p|=n$ the word length
and by 
$$\ell (p) = \sum_{i=1}^n \ell (s_i)$$
 the geometric length.   
 We write $i(p) = i(s_1)$ and $t(p) = t(s_n)$.
\end{definition}


Let $x\in\Sigma$ and define the set $\mathcal P(x,R)$ to be the set of saddle connection paths which start at $x$ and have length less than or equal to $R$. Let $\mathcal P(x)= \cup_{R > 0} \mathcal P(x,R)$ denote the set of all such paths regardless of length.

We conclude this section with two simple 
examples of translation surfaces.

\begin{example}[Slit surface]
We can consider two copies of the flat torus 
$\mathbb T^2 = \mathbb R^2/\mathbb Z^2$ with a slit removed.  In identifying the two surfaces along the slits one obtains a surface of genus $2$ with two singularities each with cone-angle $4\pi$ (coming from the ends of the slits). See Figure 2. 

\begin{figure}[h]
\centerline{
\includegraphics[width=14cm, height=6cm]{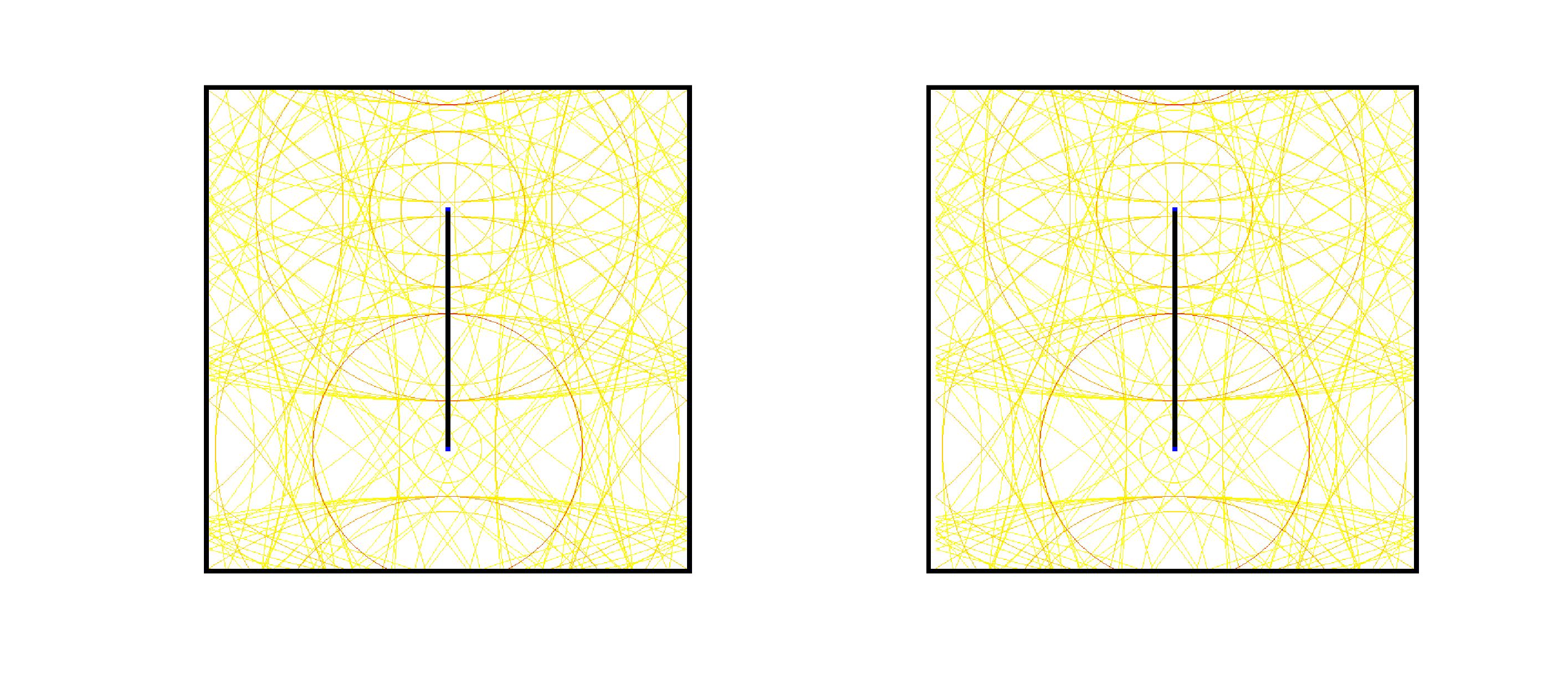}}
\caption{A circle of large radius projected onto a translation surface consisting of  two slit tori}
\end{figure}
\end{example}

\begin{example}[$L$-shaped surface]
	We can consider an $L$-shaped polygon which is a square-tiled surface where the identification of opposite sides gives a surface of genus $2$ and only one singularity (see Figure 3).  The singularity comes from the identification of the corners and has a cone-angle of $6\pi$. 
	
	Because this surface forms a ramified cover of the standard  torus,
we see that  for each coprime pair $(m,n)\in\mathbb Z^2$, there are three saddle connections whose holonomy correspond to $(m,n)$. Furthermore, we can understand the saddle connection paths on this surface by noting that a saddle connection with holonomy $(m,n)$ can concatenate with any other saddle connection with the same holonomy or two of the three saddle connections with holonomy $(m',n')\neq (m,n)$.  
\begin{figure}[h]
\centerline{
\includegraphics[width=5cm, height=5cm]{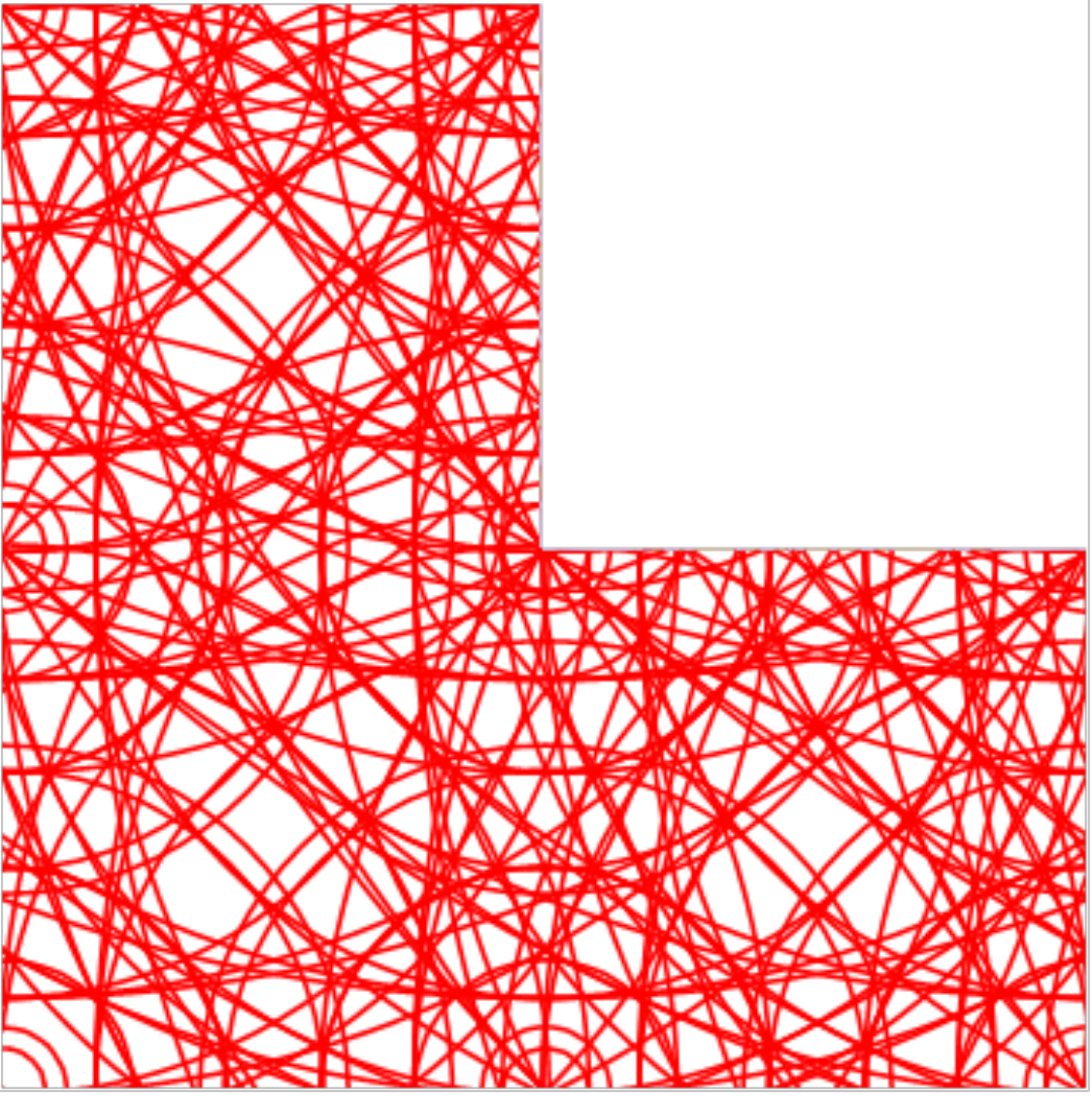}}
\caption{A large circle  projected  onto a  $L$-shaped domain.}
\end{figure}
\end{example}

\section{Proof of Theorem A}
In this section, we provide a  proof of Theorem A. This  proof follows the strategy for the asymptotic for $\hbox{Vol}_X(\mathcal B(\widetilde x,R))$ developed in \cite{cp}. From now on, fix a translation surface $X$, a singularity $x\in \Sigma$, and denote the entropy of $X$ by $h=h(X)$.
 \subsection{Large circles}  
We will first  express   the length $\ell(\mathcal C(x,R))$ of the curve $\mathcal C(x,R)$ using  saddle connection paths. For convenience, we have assumed that $x$ is a singularity; however, a similar argument would work if we considered a general point on the surface instead.
 
  The radii for the circles $\mathcal C(x,R)$  correspond to  saddle connection paths $p = (s_1, \ldots, s_n)$ which start at $x$ and are followed by a final  line segment 
  of length $r \geq 0$, say,  which begins at $t(p)$ and contains no other singularities.    In particular, this geodesic will have length 
  $\ell(s_1) + \cdots + \ell(s_n) + r$.
  \begin{lemma}\label{circleslemma}
  	  	Let $x \in\Sigma$. Then for $R>0$
  $$
  \begin{aligned}
  \ell(\mathcal C(x,R))
  &=2 \pi (k(x)+1) R 
 +\sum_{p\in\mathcal P(x,R)} 
  2\pi k(t(p))\left(R - \ell(p)\right).\cr
  \end{aligned}
  \eqno(3.1)
  $$
  
\end{lemma}

\begin{proof}
	The proof is immediate given that locally distance  minimizing geodesics starting from a singularity  are given by a saddle connection path $p = s_1s_2 \ldots s_n$ with $ \ell(s_1) + \cdots + \ell(s_n) \leq R$
	 followed by a straight line starting at the last singularity $t(p)$ of length $R - \ell(s_1) + \cdots + \ell(s_n)$.
	 In particular, this last segment contributes to a Euclidean arc of length $2\pi k(t(p))(R - (\ell(s_1) + \cdots + \ell(s_n)))$.  
	\end{proof}
It is a simple observation (for example from (3.1))
that the length $\ell (C(x, R))$ is a monotone increasing function of $R>0$.
This simple property is necessary for our method of proving the asymptotic formula in Theorem A.

\subsection{Tauberian theorem and complex functions}
The proof of Theorem A involves the application of the following classical  Tauberian theorem to the monotone and continuous function $\ell(\mathcal C(x,R))$.

 \begin{thm}[Ikehara--Wiener Tauberian theorem, cf. \cite{ellison}]\label{tauberian}\label{sec:tauberian}
Let $\rho: \mathbb R^+ \to \mathbb R^+$ be a 
  non-decreasing and right-continuous function. Formally
denote  $\eta(z):=\int_0^\infty e^{-zR}d\rho (R)$, for $z \in \mathbb C$. Assume  that $\eta(z)$ has the following properties:
 
 \begin{enumerate}
     \item there exists some $a>0$ such that $\eta(z)$ is analytic on $Re(z) > a$;
     \item $\eta(z)$ has a meromorphic extension to a neighbourhood of the  half-plane $Re(z)\geq a$;
     \item $a$ is a simple pole for $\eta(z)$, i.e.,  $C=\lim_{\epsilon\searrow 0}(z-a)\eta(z)$ exists and is positive; and 
     \item the extension of $\eta(z)$ has no poles on the line $Re(z)=a$ other than $a$.   
     \end{enumerate}
 Then $\rho(R)\sim \frac{C}{a}e^{aR}$ as $R\rightarrow \infty$.
\end{thm}

\begin{rem}
Theorem \ref{tauberian} is a reformulation of the standard Tauberian theorem for the case  $a=1$.
\end{rem}

In order to apply the above Tauberian theorem to $\ell(\mathcal C(x,R))$, we need to define and study the following complex function.

\begin{definition}
Since $\ell (\mathcal C(x, R))$ is monotone increasing we can formally  define 
the Riemann-Stieltjes integral 
$$
\eta(z) = \int_0^\infty e^{-zR} d \ell (\mathcal C(x, R))
\hbox{ for } z \in \mathbb C.  \eqno(3.2)
$$
\end{definition}

\noindent
Using  (1.1), it is easy to see that  the  complex  function $\eta(z)$ converges to an  analytic function for $Re(z) > h$.

In order to show that $\eta(z)$ satisfies the properties required to apply the Tauberian theorem, we will first rewrite $\eta(z)$ in terms of certain infinite  matrices which contain information about saddle connection paths on $X$ and their lengths. The following three properties of saddle connection paths on translation surfaces guarantee that the matrices have spectral properties which allow us to deduce that $\eta(z)$ has the required properties.
\begin{lemma}\label{lem:transhyp}
For all translation surfaces $X$, with corresponding oriented saddle connection set $\mathcal S$, the following statements hold.
\begin{enumerate}
\item
For all $\sigma > 0$ we have $\sum_{s \in \mathcal S} 
e^{-\sigma \ell(s)} < \infty$; 
\item For any directed saddle connections $s, s' \in \mathcal S$ there exists a saddle connection path $p$ beginning with $s$ and ending with $s'$; and
\item There does not exist a
$d > 0$ such that 
$$
\{ \ell (c) \hbox{ : } c \hbox{ is a closed saddle connection path} \} \subset d \mathbb N.
$$
\end{enumerate}
\end{lemma}

These simple observations  are taken from  \cite{cp}.

\medskip
We now turn to expressing $\eta(z)$ in terms of infinite matrices in order to obtain the desired meromorphic extension of $\eta(z)$. 

\subsection{Infinite matrices}
We will first show that $\eta(z)$ can be written in terms of saddle connection paths on $X$. We will then rewrite $\eta(z)$ in terms of countably infinite matrices which capture the saddle connection path data of $\mathcal P(x)$.

\begin{lemma}
For $Re(z) > h$,  we have that 
$$
\begin{aligned}
\eta(z) &=
\frac{2\pi}{z}(k(x) + 1) + 
  \frac{2\pi}{z} k(t(p))
\sum_{p \in\mathcal P(x) }
e^{- z \ell (p)}.
\end{aligned}
\eqno(3.3)
$$

\end{lemma}

\begin{proof}
We begin with the contribution to $\ell (\mathcal C(x,R))$ that comes from the Euclidean circle of radius $R$ and cone-angle 
$2\pi (k(x) + 1)$. 
This involves  considering 
$$
2\pi (k(x) + 1)\int_{0}^\infty e^{-zR} d R  = 
- \frac{2\pi}{z}(k(x) + 1) \left[  e^{-zR} \right]_0^\infty =  \frac{2\pi}{z}(k(x) + 1).
$$
We see from (3.1) that  the  Euclidean circles 
with cone-angle $2\pi k(t(p))$
centered at $t(p)$ 
following  the saddle 
connection path $p = (s_1, \ldots, s_n)$ will have radius  $R - \ell(p)$.  The contribution to $\eta(z)$ from saddle connection path $p$ will be
$$
\begin{aligned}
2\pi k(t(p))\int_{\ell(p)}^\infty e^{-zR} d\left(R-\ell(p) \right) 
& = 
2\pi k(t(p))
e^{-z\ell(p) }\int_{0}^\infty e^{-zR} d R \cr
&= \frac{2\pi  k(t(p))}{z} e^{-z\ell(p)}\cr
\end{aligned}
$$
by the change of variable $T = R-\ell(p)$.    
By (1.1) we have that the summation over all of  these contributions is uniformly convergent for $Re(z) > h$.
This gives the required result.
\end{proof}

We will next rewrite $\eta(z)$ using the following family of infinite matrices $M_z$ ($z\in \mathbb C$).

\begin{definition} 

For $z\in \mathbb C$ with $Re(z)>0$, we define  the infinite matrix $M_z$,  with rows and columns  indexed by $\mathcal S$, by 
$$
M_z(s,s') =
\begin{cases}
 e^{-z \ell(s')} & \hbox{ if  $ss'$ is an allowed  geodesic path,} \cr
 0 & \hbox{ otherwise }
 \end{cases}
$$
where the rows and columns are indexed by saddle connections partially ordered by their lengths.

\end{definition}
\medskip

The saddle connection path length data for $X$ can be retrieved from these matrices in the following way. Let $\mathcal P(n,s,s')$ denote the set of saddle connection paths consisting of $n$ saddle connections, 
starting with $s$ and ending with $s'$, where $s,s'\in\mathcal S$. It then follows from 
formal matrix multiplication that for any $n\geq 1$, the $(s,s')^{th}$ entry of the $n^{th}$ power of the matrix is given by
$$
M_z^n(s,s') = e^{z\ell(s)}\sum_{p\in \mathcal P(n+1,s,s')} e^{-z\ell(p)}.\eqno(3.4)
$$
In order to rewrite $\eta(z)$ in terms of these  matrices, we need to consider the matrices' associated bounded linear operators $\widehat M_z:\ell^\infty \rightarrow \ell^\infty$ which act on the Banach space 
$$\ell^\infty = \left\{\underline u = (u_s)_{s\in \mathcal S} \hbox{ : } u_s \in \mathbb C, 
\quad
\sup_{s\in \mathcal S}|u_s| <\infty\right\},$$ with the norm $\|\underline u\|_{\ell^\infty} = \sup_{n}|u_n| $. The linear operators $\widehat M_z:\ell^\infty\rightarrow \ell^\infty$ are defined
by 
$$(\widehat M_z \underline u )_s = \sum_{s' \in \mathcal S} M_z(s,s') u_{s'} \quad (s \in \mathcal S).$$ Note that $\widehat M_z$ is bounded by Property (1) of Lemma \ref{lem:transhyp}.

Using these operators, the expression for $\eta(z)$ in (3.3) and the expression of saddle connection paths in terms of the matrices (equation (3.4)), for $Re(z)>h$, we can write
$$
\begin{aligned}
\eta(z) 
&= 
\frac{2\pi}{z}(k(x) + 1) + 
  \frac{2\pi}{z} k(t(p))
\sum_{p \in\mathcal P(x) }
e^{- z \ell (p)}\cr
&=\frac{2\pi}{z} (k(x)+1)
+ \frac{2\pi}{z} 
 {\underline v(z)\cdot \left( \sum_{n=0}^\infty \widehat M_z^n\right)}\underline u(z)\cr
 &= \frac{2\pi}{z} (k(x)+1)
+ \frac{2\pi}{z} 
 {\underline v(z) \cdot \left( I - \widehat M_z\right)^{-1}}\underline u(z),
\end{aligned}\eqno(3.5)
$$
where we denote
 $$\underline u(z) 
 =  (k(t(s)))_{s} \in  \ell^\infty
\hbox{  and } \underline v(z) 
 =  
 (\chi_{\mathcal E_x}(s)
 e^{-z\ell (s)})_{s}  \in  \ell^1,$$
  where 
 $\chi_{\mathcal E_x}$ denotes the characteristic function of the set 
 $\mathcal E_x = \{s\in\mathcal S \hbox{ : } i(s) =x\}$
 of oriented saddle connections starting from the singularity $x \in \Sigma$.
 
Next, we make use of an idea developed in \cite{hk} which allows us relate the invertibility of $(I-\widehat M_z)$ to the spectra of a family of finite matrices.
To this end, we note that we can write $M_z$ as 
$$
M_z = 
\left(
\begin{matrix}
A_z&U_z  \\ V_z &B_z
\end{matrix}
\right),
$$
where $A_{z} $ is the $k\times k$ finite sub-matrix of $M_{z}$ corresponding to the first $k\in\mathbb N$, say, oriented  saddle connections and the other sub-matrices $B_{z} ,U_{z} ,V_{z} $ are infinite.

Note that given $\epsilon>0$, $(I-B_z)$ is invertible for $k$ sufficiently large, by Property (1) of Lemma \ref{lem:transhyp}. Hence, for such $k$, the finite matrix $W_z:=A_{z} +U_{z} (I-B_{z} )^{-1}V_{z} $ exists.

By formal matrix multiplication, one can check that whenever $\det(I-W_z)\neq 0$, $I- {M_z}$ is invertible  with inverse
$$(I- {M_z})^{-1}=\left(
\begin{matrix}
I &0  \\ (I -  {B_z})^{-1}{V_z} & (I -  {B_z})^{-1}
\end{matrix}
\right)
\left(
\begin{matrix}
(I- {W_z})^{-1}&(I- {W_z})^{-1} {U_z} (I- {B_z})^{-1}  \\ 0 &I 
\end{matrix}
\right),\eqno(3.6)
$$
on $Re(z)>\epsilon$ for $k$ sufficiently large.\\
Using the  factorization in (3.6), we can deduce the following lemma.

\begin{lemma}
Fix $\epsilon>0$. Then $\eta(z)$ has a meromorphic extension to $Re(z)>\epsilon$ of the form
$$\eta(z)=\frac{\phi(z)}{\det(I-W_z)},$$
where $\phi(z)$ is analytic on $Re(z)>\epsilon$ and
 $k$ chosen to be sufficiently large, where $k$ denotes the size  of the $k\times k$ matrix $W_z$. 
\end{lemma}
 Note that the poles of this extension occur for $z$ such that 1 is an eigenvalue of the matrix 
  $W_z$. By Property (2) of Lemma \ref{lem:transhyp}, it follows that the $W_z$ are irreducible matrices (see \cite{cp} for details).

 \begin{lemma}\label{thmAextn}
 The meromorphic extension of $\eta(z)$ satisfies the assumptions of the Tauberian theorem. In particular, the meromorphic extension of the $\eta(z)$ is analytic for $Re(z) > h$, with a simple pole at $z=h$ which has positive residue, and there are no other poles on the line $Re(z)=h$.
 \end{lemma}

\begin{proof}
The fact that $\eta(z)$ is analytic on $Re(z)>h$ and has a singularity at $z=h$ follows from (1.1).
  
  The pole at $z=h$ corresponds to the matrix $W_h$ having $1$ as an eigenvalue.
  To show the simplicity  of the pole it suffices to show that for $\sigma>0$, the maximal eigenvalue $\lambda(\sigma)$ for $W_\sigma$
satisfies   $\frac{\partial \lambda(\sigma)}{\partial \sigma}|_{\sigma=h}\neq 0$.  However, we can write
 $\frac{\partial \lambda(\sigma)}{\partial \sigma}|_{\sigma=h}
= u. W_{h}'v < 0$ where:   $u, v>0$ are  the normalized left and right eigenvectors, respectively,  of the positive matrix $W_{h}$; and $W_{h}'$ is the matrix with entries 
$W_{h}'(i,j) = \frac{\partial W_{\sigma}}{\partial \sigma}|_{\sigma =h} < 0$
(cf. proof of Lemma 4.3 of \cite{cp}).
  
 

It remains to show that there are no other poles on the lines $Re(z)=h$. This follows from comparing the absolute value of the diagonal entries of powers of $W_h$ and $W_{h+iy} $, for $y\in\mathbb R$.  In particular, such entries are Dirichlet series  containing terms of the form $e^{-h\ell(q)}$ and $e^{-(h+iy)\ell(q)}$, respectively, corresponding to closed geodesics $q$. It follows from Wielandt's theorem for matrices, that the only way for $W_{h+iy} $ and $W_{h} $ to both have 1 as an eigenvalue, is if either $y=0$ or for Property (3) of Lemma \ref{lem:transhyp} to not hold.\end{proof}


Finally, because $\eta(z)$ satisfies the properties of the Tauberian theorem, the proof of Theorem A follows from the application of the Tauberian theorem to the function $\ell(\mathcal C(x,R))$,  i.e. $\ell(\mathcal C(x,R))\sim \frac{C}{h}e^{hR}$ where $C>0$ is the residue of $\eta(z)$ at $z=h$.

\medskip
 
 We conclude this section by presenting an analogous counting result that we will use later.  
Let  $N(x,R)$  denote   the number of saddle connection paths starting at  $x$ and of length less than or equal to $R$, i.e. $N(x,R)=\#\mathcal P(x,R)$.

 \begin{prop}\label{N(x,R)}
For each  $x \in \Sigma$ there exists $E = E(x)>0$ so that
 $N(x,R)\sim (E/h)e^{hR}$ as $R \to \infty$.
 \end{prop}
 
\begin{proof}
The proof is  completely analogous to that of Theorem A.
We can write 
 $$
 \begin{aligned}
 \eta_N(z):=\int_0^\infty e^{-zR} dN(x,R)
 &=\sum_{p\in\mathcal P(x)}e^{-z\ell(p)}
 &= \underline v(z)\cdot \left( \sum_{n=0}^\infty \widehat M_z^n \right) \underline 1,
 \end{aligned}
 $$
 where $(\underline 1)_s=1$ for all $s\in\mathcal S$  and $\underline v(z) 
 =  
 (\chi_{\mathcal E_x}(s)
 e^{-z\ell (s)})_{s}  \in  \ell^1$ as before.
 Because 
 of the spectral properties of the finite matrices $W_z$ associated to the $M_z$, $\eta_N(z)$ satisfies the assumptions of
 Theorem \ref{tauberian}, 
  hence $N(x,R)\sim (E/h)e^{hR}$ 
 where  $E>0$ is the residue of the pole at $z=h$ for $\eta_N(z)$.
\end{proof}

\section{Distribution of large circles}
We briefly describe the  strategy for the proof of Theorem B,  which states that the sequence of probability measures
$\mu_R$
 defined by
$$\mu_R(A)=\frac{\ell(\mathcal C(x,R)\cap A)}{\ell(\mathcal C(x,R))},\hbox{ for Borel sets $A\subset X$,}$$ converges in the weak-star topology to some probability measure $\mu$ as $R$ tends to infinity.

The proof of 
this theorem comes from  asymptotic results for $\ell(\mathcal C(x,R)\cap B)$, for (small) balls $B\subset X\backslash \Sigma$. In particular, we  will  show that for a (small) ball $B\subset X\backslash \Sigma$, there exists some constant $C(B)> 0$ such that
$$\ell(\mathcal C(x,R)\cap B)\sim C(B)e^{hR} \hbox{ as } R \to \infty.$$
Combining this asymptotic with Theorem A, we can then deduce that for Borel sets $A$ with $\mu(\partial A) = 0$ we have 
$$\lim_{R\rightarrow \infty} \mu_R(A)=\lim_{R\rightarrow \infty} \frac{\ell(\mathcal C(x,R)\cap A)}{\ell(\mathcal C(x,R))}=\lim_{R\rightarrow \infty}\frac{C(A)e^{hR}}{(C/h)e^{hR}}=\frac{C(A)}{(C/h)},$$ which implies $\mu_R$ converges to $\mu:=C(A)/(C/h)$ in the weak-star topology (see \cite{walters}).
Finally, we need  to check that $\mu(A)$ defines a probability measure.

In order to deduce the limit  above  for the functions $\ell(\mathcal C(x,R)\cap B)$, it might seem natural to  consider an appropriate complex function for $\ell(\mathcal C(x,R)\cap B)$ and then apply the Tauberian theorem, as in the proof of Theorem A.
 However, due to the fact that $\ell(\mathcal C(x,R)\cap B)$ may not be monotonic, we cannot directly 
 apply the Tauberian theorem.
 Therefore, we will instead prove an asymptotic result for the non-decreasing continuous function $V_A(R):=\hbox{Vol}_{\widetilde{X}} (\mathcal B(\tilde{x},R)\cap \widetilde A)$, for Borel sets $A\subset X$ i.e.,  the area of the intersection of a ball of radius $R$ in the universal cover of $X$ intersected with the lifts of $A$. We will then be able to use these asymptotics to indirectly deduce the corresponding asymptotics for $\ell(\mathcal C(x,R)\cap B)$ for balls $B\subset X\backslash \Sigma$.

\subsection{Notation}

 We begin by introducing  some useful notation.
Let  $\pi : \widetilde X \to X$ denote   the canonical projection from the universal cover   $\widetilde X$
to $X$.  Let $\widetilde \Sigma$ and $\widetilde {\mathcal S}$ 
denote  the lifts of the singularity sets  $\Sigma$ and oriented saddle connections $\mathcal S$, respectively. Fix a lift $\widetilde x\in \widetilde \Sigma$ of $x$. Let $p\in\mathcal P(x)$ be a saddle connection path with lift
$\widetilde {p}$. We denote the length of $\widetilde {p}$ by $\ell(\widetilde p)$ and write $i(\widetilde {p}), t(\widetilde {p})  \in \widetilde \Sigma$ for the singularities at the beginning and end of $\widetilde p$, respectively.

\begin{definition} 
Let  $z \in \Sigma$ with a choice of  lift $\widetilde z \in \widetilde \Sigma$
(i.e., $\pi(\widetilde z) = z$) and let $R >0$.
\begin{enumerate}
\item
We denote   a Euclidean disk 
$\mathcal D(\widetilde z, R)$
 in $\widetilde X$
(with center $\widetilde z$ and radius $R>0$) 	by
$$\mathcal D(\widetilde z, R) \subset \widetilde X$$
consisting of the 
set  of those points $\widetilde y \in \widetilde X$ which are  joined to $\widetilde z$ by a straight line segment 
of length {\it at most} $R>0$, which does not have a singularity in its interior.
\item
Let $p\in\mathcal P(x)$ be a saddle connection path with unique lift $\widetilde p$ based at $\widetilde x$. Let $\widetilde w:=t(\widetilde p)\in\widetilde\Sigma$.
We define  a 	 Euclidean sector 
$$\mathcal E({p}, R) \subset 
\mathcal D(\widetilde w, R) \subset 
\widetilde X$$
 associated to $ {p}$ 
 (with center $\widetilde w$ and radius $R>0$) by the
set  of points $y \in \widetilde X$ which are  joined to $\widetilde w$ by a straight line segment 
of length {\it at most} $R>0$,  
which does not have a singularity in its interior, 
 and which additionally  forms a geodesic in $\widetilde X$ with $\widetilde{ p}$.

\end{enumerate}
\end{definition}  

 On occasion  it will be convenient to consider  sectors on $X$, which we define in an analogous way and denote by $\mathcal E_X(p,R)$.

Given a radius $R>0$, we can write the ball
$ \mathcal B(\widetilde x, R)$ in $\widetilde X$ 
 as 
$$ \mathcal B(\widetilde x, R):= \mathcal D(\widetilde x,R) \cup 
\bigcup_{p\in\mathcal P(x,R)} \mathcal E( {p}, R - \ell( {p}) ).$$

Fix a Borel set $A\subset X$. As mentioned at the beginning of this section, we are interested in an asymptotic result for $V_A(R):=\hbox{Vol}_{\widetilde{X}} (\mathcal B(\tilde{x},R)\cap \widetilde A)$. We will proceed with a similar approach to the one we used for Theorem A by making the following observation which can be compared to Lemma 3.1.

 \begin{lemma}
 For $R>0$ we can write, 
 	$$
 	\begin{aligned}
	V_A(R)
= \hbox{Vol}_{\widetilde{X}} (\mathcal D(\widetilde x, R)\cap \widetilde{A})+ 
 	\sum_{p\in\mathcal P(x,R)}
 	\hbox{Vol}_{\widetilde{X}} ({\mathcal E}( {{p}}, R - \ell(p))\cap \widetilde A).
 	 \end{aligned}
 		\eqno(4.1)
 	$$
	where the the first term is the volume of the Euclidean disk and the second term is expressed in terms of the volumes of 
	Euclidean sectors.
 \end{lemma}
 
 The heuristic of the basic identity (4.1) is illustrated in Figure 4.
 

   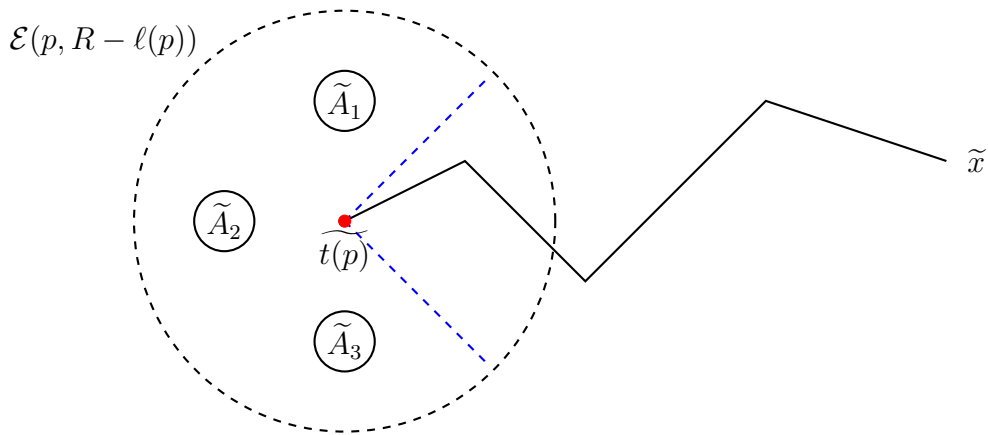
\begin{figure}[h]\label{fig:sector}
\centerline{
\begin{tikzpicture}[scale=0.8]
\draw [thick] (0,0) -- (2,1)--(4,-1)--(7,2)--(10,1);
\draw [thick, dashed,blue] (0,0) --(2.4,2.4);
\draw [thick, dashed,blue] (0,0) --(2.4,-2.4);
\draw [thick, dashed] (0,0) circle [radius=3.5];
\draw [fill, red] (0,0) circle [radius=0.1];
\node at (10.5, 1) {$\widetilde x$};
\node at (-4, 3) {${{\mathcal E}(   {p},R - \ell(p))}$};
\node at (0, -0.5) {$\widetilde {t(p)}$};
\draw [thick] (0,2) circle [radius=0.5];
\draw [thick] (-2,0) circle [radius=0.5];
\draw [thick] (0,-2) circle [radius=0.5];
\node at (0,2) {$\widetilde A_1$};
\node at (-2,0) {$\widetilde A_2$};
\node at (0,-2) {$\widetilde A_3$};
\end{tikzpicture}
}
\caption{The ball  $\mathcal B(\widetilde x, R) \subset \widetilde X$ is a union of appropriate 
Euclidean sectors $\mathcal E(\cdot, \cdot)$ centered at lifts of singularities. We are interested in the volume of the lifts of $A$, represented by $\widetilde A_1$, $\widetilde A_2$ and $\widetilde A_3$, which intersect $\mathcal E(p, R-\ell(p))$.}
\end{figure}

\begin{example}
In the particular case that $A=X$, the identity (4.1) reduces to 
$$	
\begin{aligned}
	V_X(R)&= 
(k(x) + 1) \pi R^2
 	+\sum_{p\in\mathcal P(x,R)}
k(t(p))\pi (R - \ell(p))^2.
\end{aligned}
 $$
\end{example}

   \subsection{Asymptotic formula for $V_A(R)$}
 In order to derive an asymptotic formula for $V_A(R)$
   we can now proceed by analogy with the proof of Theorem A. 
To begin, we generalize the definition of  $\eta(z)$  as follows.

\begin{definition}
For Borel sets $A\subset X$ with $\hbox{Vol}_X(A)>0$, we can formally  define a complex function by the Riemann-Stieltjes integral
$$
\eta_A(z) = \int_0^\infty e^{-zR} d 
 V_A(R),
\hbox{ for } z \in \mathbb C.  \eqno(4.2)
$$
\end{definition}

We want to show that the growth rate of $V_A(R)$ is positive. First note that if $\hbox{\rm Vol}_X(A)=0$ then $V_A(R)=0$ for all $R>0$. Before we proceed, we require the following lemma (a similar result can be found in \cite{dankwart}).

\begin{lemma}\label{non-zero}
Let $A\subset X$ be a Borel set such that $\hbox{\rm Vol}_X(A)>0$. Let 
$\hbox{\rm diam}(X)$ denote the finite  diameter of $X$. Then there exists a saddle connection $s'\in\mathcal S$ such that 
$$\hbox{\rm Vol}_{X}(\mathcal E_X(s',\hbox{\rm diam}(X))\cap  A)>0.$$
\end{lemma}

\begin{proof}
We require two simple preliminary results.

\bigskip
\noindent
{\it Claim 1.} For any $x\in X$ there is a straight line  segment $g_x$
 joining $x$ to some singularity $y :=t(g_x) \in \Sigma$ of length at most $\hbox{\rm diam} (X)$.

\bigskip
\noindent
 {\it Proof of claim 1.}  A  translation surface is  a geodesic space of finite diameter. In particular, we can connect $x$ to a singularity in $X$ by a  geodesic which necessarily takes the form $p_x=g_xs_1\ldots s_n$ or $p_x=g_x$ of length $\ell(p_x)\leq$ diam$(X)$, where the $s_i$ are oriented saddle connections and $g_x$ is an oriented straight line segment from $x$  to some  some singularity $t(g_x) \in \Sigma$. In either case, $g_x$ is the required straight line segment.   \qed 



\bigskip
\noindent
{\it Claim 2.}  Let $a\in A$. By claim 1, there exists an oriented straight line segment $g_a$ connecting $a$ to some singularity $t(g_a)$. The sector $\mathcal E_X(g_a,2\hbox{\rm diam}(X))$ must contain a singularity.

\bigskip
\noindent
 {\it Proof of claim 2.} 
 Assume for  a contradiction that
  $\mathcal E_X(g_a,2\hbox{\rm diam}(X)) \cap \Sigma = \emptyset$.
  Since the angle of the sector is greater than or equal to $2\pi$
  and  by assumption $\mathcal E_X(g_a,2 \hbox{\rm diam}(X))$ is Euclidean,
   one can choose 
   a ball $\mathcal B(c,\hbox{diam}(X)) \subset \mathcal E_X(g_a,2\hbox{diam}(X))$ centered at 
   $c \in \mathcal E_X(g_a,2\hbox{diam}(X))$ and of diameter 
   $\hbox{diam}(X)$ (see Figure 5). 
   However, by claim 1, there exists a straight line segment $g_c$ of length $\ell(g_c)\leq$ diam$(X)$,  connecting $c$ to some  singularity $z\in \Sigma$. This implies that $z\in \mathcal B(c,\hbox{diam}(X))\subset \mathcal E_X(g_a,2\hbox{diam}(X))$ which gives a contradiction. \qed 
   
  \medskip 

We can now complete the proof of  the lemma.  For any $a \in A$, claim 2 
implies we can choose $z \in \mathcal E_X(g_a,2\hbox{diam}(X)) \cap \Sigma$. Thus we can choose an oriented saddle connection $s_a$ of length $\ell(s_a)\leq 2\hbox{diam}(X)$, from $z$  to $y = t(g_a)$ and such that  $s_ag_a^{-1}$ is an allowed geodesic. Because $\ell(g_a^{-1}) = \ell(g_a)\leq \hbox{diam}(X)$, it follows that $a\in \mathcal E_X(s_a,\hbox{diam(X)})$.

Finally, since  for all $a\in A$ we have $\ell(s_a)\leq 2\hbox{diam}(X)$, the set $\{s_a\}_{a\in A}$ is necessarily  finite. Therefore, because $\hbox{Vol}_X(A)>0$ and $\bigcup_{\{s_a:a\in A\}}\mathcal E_X(s_a,\hbox{diam}(X)\cap A)=A$, at least one of the finite number of sectors, $\mathcal E_X(s_a,\hbox{diam}(X))$, must satisfy $\hbox{Vol}_X(\mathcal E_X(s_a,\hbox{diam}(X)\cap A)>0$.
\end{proof}

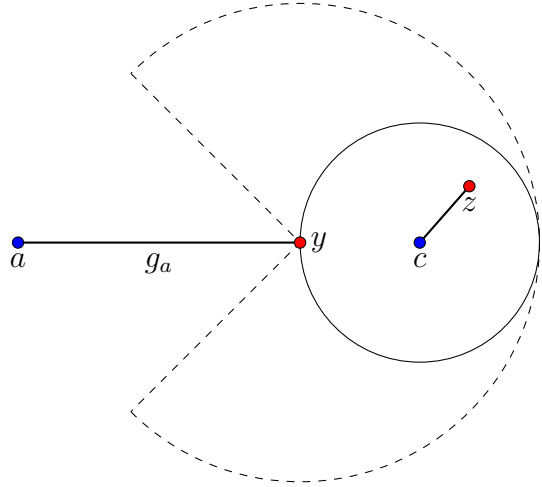
\begin{figure}[h!]
    \centering
\begin{tikzpicture}[scale=0.75]
    \draw[thick] (-1,0)--(4,0);
    \draw[dashed] (1,3)--(4,0)--(1,-3);
    \draw[dashed] (1,-3) arc (-135:135:4.243);
    \draw (4+2.1215,0) circle (2.1215);
    \draw[thick] (4+2.1215,0) -- (7,1);
    \draw[fill=blue] (-1,0) circle (0.1);
        \draw[fill=blue] (4+2.1215,0) circle (0.1);

    \draw[fill=red] (4,0) circle (0.1);
    \draw[fill=red] (7,1) circle (0.1);
    
    \node[below] at (-1,0) {$a$};
   \node[below] at (1.5,0) {$g_a$};
    \node[right] at (4,0) {$y$};
    \node[below] at (4+2.1215,0)  {$c$};
        \node[below] at (7,1)  {$z$};

\end{tikzpicture}    \caption{The point $a\in A$ is connected to $y\in\Sigma$ via segment $g_a$. The dotted sector represents $\mathcal E_X(g_a,2\hbox{diam}(X))$. The ball $\mathcal B(c,\hbox{diam}(X))\subset \mathcal E_X(g_a,2\hbox{diam}(X))$  contains a  singularity $z\in\Sigma$. }
    \label{fig:sectorcontradiction}
\end{figure}


\begin{lemma}\label{h_A(X)>0}
Let $A\subset X$ be a Borel set such that $\hbox{Vol}_X(A)>0$. Then 

$$\lim_{R\rightarrow \infty} \frac{1}{R}\log(V_A(R))=h.$$
\end{lemma}
\begin{proof}
We will prove the result by considering upper and lower bounds for $V_A(R)$ and their logarithmic limits. For the upper bound it suffices to use $ V_A(R)  \leq V_X(R)=\hbox{Vol}_{\widetilde X}(\mathcal B(\widetilde x,R))$ and (1.2).
For the lower bound, observe that

$$
\begin{aligned}
V_A(R)
 	&\geq   \hbox{Vol}_X(\mathcal E(s',\hbox{diam}(X))\cap A)\cdot N(x,s',R-\hbox{diam}(X)),
 	 \end{aligned}
 	 $$
 	 where $N(x,s',R)$ denotes the number of saddle connection paths starting at $x$ ending with saddle connection $s'$ and of length less than or equal to $R$.
 	 
 	 Next we recall a result from Dankwart \cite{dankwart} which states that any two oriented saddle connections $s_1,s_2$ can be connected by a third oriented saddle connection $s$ of length smaller than a given $L>0$, such that the path $s_1ss_2$ form an allowed  saddle connection path. 
 	 Using this result, we see  that $N(x,s',R)\geq N(x,R-(L+\ell(s'))$ and hence 
 	 $$V_A(R)\geq \hbox{Vol}_X(\mathcal E(s',\hbox{diam}(X))\cap A)\cdot N(x,R-(\hbox{diam}(X)+L+\ell(s')).$$
Finally, we complete the proof by recalling (1.2)  and Proposition \ref{N(x,R)} and then

$$\lim_{R\rightarrow \infty} \frac{1}{R}\log(V_X(R))=\lim_{R\rightarrow \infty} \frac{1}{R}\log(N(x,R)) =h$$
as required.
\end{proof}

We next show that Lemma \ref{h_A(X)>0} can be improved to an asymptotic formula by employing   the method used to prove Theorem A.

\begin{prop}
If $\hbox{\rm Vol}_X(A) > 0$
then there exists $C(A) > 0$ such that 
$$
V_A(R)
 \sim  \frac{C(A)}{h} e^{hR} \hbox{ as } R \to \infty.
$$
\end{prop}


\begin{proof}
First note that by Lemma \ref{h_A(X)>0}, the assumption that $\hbox{Vol}_X(A)>0$ implies that the complex function $\eta_A(z)$  has a pole at $z=h$ and converges  to an analytic function for $Re(z)>h$.
In particular, for $Re(z) > h$ we can can use (4.1) to write
 	$$
 	\begin{aligned}
	\eta_A(z) &= \int_0^\infty e^{-zR}
	d V_A(R)\cr
&=  z \int_0^\infty \hbox{Vol}_{\widetilde X}(\mathcal D(\widetilde x,R)\cap \widetilde A) e^{-zR} dR
 + z \sum_{p\in\mathcal P(x)}
 \int_{\ell(p)}^\infty 
 \hbox{Vol}_{\widetilde X}(\mathcal E({p},R-\ell(p))\cap \widetilde A)e^{-zR} dR\cr
  &=  z  \int_0^\infty  \hbox{Vol}_{\widetilde X}(\mathcal D(\widetilde{x},R)\cap \widetilde A) e^{-zR} dR 
+ z  \sum_{p\in\mathcal P(x)}
 e^{-z \ell(p))}
 \int_0^\infty
 	 \hbox{Vol}_{\widetilde X}(\mathcal E({p},r)\cap \widetilde A)  e^{-z r} dr     
 	 \end{aligned}
 	$$
using the change of variables $r = R- \ell(p)$ for each of the terms in the final summation.
By using the matrices $M_z$, we can write $\eta_A(z)$ as 
$$
\begin{aligned}
\eta_A(z) 
 &=
 z  \int_0^\infty  \hbox{Vol}_{\widetilde X}(\mathcal D(\widetilde{x},R)\cap \widetilde A) e^{-zR} dR 
+
 z\underline v(z)\cdot \left(\sum_{n=0}^\infty \widehat M_z^n\right)\underline u_A(z) \cr
 &=
 z  \int_0^\infty  \hbox{Vol}_{\widetilde X}(\mathcal D(\widetilde{x},R)\cap \widetilde A) e^{-zR} dR 
+
 z\underline v(z) \cdot\left(I  - \widehat M_z\right)^{-1}\underline u_A(z) \cr
\end{aligned}
$$
where
\begin{enumerate}

\item[a)]
 $\underline u_A(z) 
 = 
 \left(
 \int_0^\infty\hbox{Vol}_{\widetilde X}(\mathcal E({p},r)\cap \widetilde A)  e^{-z r}
dr\right)_{s \in \mathcal S} \in \ell^\infty$; and 
 \item[b)] $\underline v(z) 
 =  
 (\chi_{\mathcal E_x}(s)
 e^{-z\ell (s)})_{s\in \mathcal S} \in \ell^1$, where 
 $\chi_{\mathcal E_x}$ denotes the characteristic function of the set 
 $\mathcal E_x = \{s \in\mathcal S\hbox{ : } i(s) =x\}$.
 \end{enumerate}
The quadratic growth of the volume function 
$R \mapsto \hbox{Vol}_{\widetilde X}(\mathcal D(\widetilde{x},R)\cap \widetilde A)$
gives that the term
 $$ z  \int_0^\infty  \hbox{Vol}_{\widetilde X}(\mathcal D(\widetilde{x},R)\cap \widetilde A) e^{-zR} dR $$ 
 is analytic for $Re(z) > 0$.
Moreover, the   sequences $\underline v(z)$ and $\underline u_A(z)$
  are analytic on $Re(z)>0$. Furthermore, by Lemma \ref{non-zero}, $\underline u_A(h)$ is non-zero.
It follows from the proof of Theorem A (or \cite{cp}), that the  complex function $\eta_A(z)$ has the following properties:
\begin{enumerate} 
\item $\eta_A(z)$ converges to a non-zero analytic function for $Re(z) > h$;
\item 
 $\eta_A(z)$ extends to a simple pole at  $z = h$
with residue $C(A)>0$;  and 
 \item $\eta_A(z)$ has an analytic extension 
 to a neighbourhood of 
 $$\{z \in \mathbb C \hbox{ : }Re(z) > h\} - \{h\}.$$
\end{enumerate}

Finally,  we can apply Theorem \ref{tauberian} to the monotone continuous function $V_A(R)$
  to deduce the asymptotic formula
$$
V_A(R)
 \sim  \frac{C(A)}{h}e^{hR} \hbox{ as } R \to \infty,\eqno(4.4)
$$
where $C(A)>0$ is the residue of $\eta_A(z)$ at $z=h$.
This completes the proof of the proposition.
\end{proof}



Recall that $V_X(R)\sim (C(X)/h)e^{hR}$ and $\ell(\mathcal C(x,R))\sim (C/h)e^{hR}$ for some constants $C(X),C>0$. 
We conclude this section with the following relation between $C$ and $C(X)$ which will be used in the proof of Theorem B.

\begin{lemma}\label{lem:coeff}
$C/h=C(X).$
\end{lemma}

\begin{proof}
The coefficients $C$ and $C(X)$ are obtained as the residues of $z=h$ for $\eta_A(z)$ and $\eta_X(z)$, respectively. In particular, using Example 4.3
we see that
 $$
 \begin{aligned}
 C(X)&=\lim_{z\rightarrow h} (z-h) \frac{2\pi}{z^2} 
 {\underline v(z) \cdot \left( I - \widehat M_z\right)^{-1}}\underline u(z)\cr
 &=\frac{1}{h}\lim_{z\rightarrow h} (z-h) \frac{2\pi}{z} 
 {\underline v(z) \cdot \left( I - \widehat M_z\right)^{-1}}\underline u(z)\cr
 &=\frac{C}{h}.
 \end{aligned}
 $$
\end{proof}


\section{Proof of Theorem B}
We now have all the ingredients to complete  the proof of Theorem B. Recall that in the previous section we showed that if $\hbox{Vol}_X(A)>0$ then  $V_A(R)\sim (C(A)/h)e^{hR}$ for some $C(A)>0$ and if $\hbox{Vol}_X(A)=0$ then for all $R>0$, $V_A(R)=0$ (and so we can formally write $V_A(R)\sim 0e^{hR}$). We use this to define the measure $\mu$ as follows: for all Borel sets $A\subset X$, we define
$$\mu(A)=\begin{cases} \frac{C(A)}{C/h}=\frac{C(A)}{C(X)}&\text{ if $\hbox{Vol}_X(A)>0$}\\
0&\text{ if $\hbox{Vol}_X(A)=0$,}

\end{cases}$$

where the equality comes from Lemma 4.8.

By using a similar expression for the residues used in the proof of Lemma \ref{lem:coeff}, one can check that $\mu$ defines a probability measure on $X$.

Furthermore, it is easy to see that $\mu$ is absolutely continuous with respect to the volume measure on $X$, $\hbox{Vol}_X$, (i.e. $\mu(A)=0$ if and only if $\hbox{Vol}_X(A)=0$ for all Borel sets $A$). In particular, if $\hbox{Vol}_X(A)=0$ then $V_A(R)=0$ for all $R>0$ and $\mu(A)=0$. For the case where $\mu(A)=0$, we can consider the contrapositive statement and observe that if $\hbox{Vol}_X(A)>0$ then we have shown that $\mu(A)=C(A)/C(X)>0$.

It remains to show that $\mu_R\rightarrow \mu$ in the weak-star topology.  To this end it suffices to show that 
$\mu_R(B)$ converges to $\mu(B)$ for appropriately  small balls  $B \subset  X \backslash \Sigma$ (see \cite{walters}).

The proof of  Theorem B now comes in two steps.  The first step is to deduce  an asymptotic  result for annuli.  The second 
step is to let the thickness of the annuli tend to zero. 

To achieve the first step,  given $\epsilon > 0$ we  denote by 
$$
\mathcal A(\widetilde {x}, R-\epsilon, R) := \mathcal B(\widetilde {x}, R) - \mathcal B(\widetilde {x}, R-\epsilon),
\quad\hbox{ for } \widetilde x \in \widetilde \Sigma \hbox{ and } R>0,
$$
the corresponding annulus.
We can then use (4.4) twice 
to deduce an asymptotic expression for $\hbox{Vol}_{\widetilde X}(\mathcal A(\widetilde x,R-\epsilon,R)\cap \widetilde B)$ of the form
$$
\hbox{Vol}_{\widetilde X}(\mathcal A(\widetilde x,R-\epsilon,R)\cap \widetilde B)
	\sim \frac{C(B)}{h}e^{hR}\left(1 - e^{-h\epsilon} \right) \hbox{ as } R \to \infty,
\eqno(5.1)$$
where $ \widetilde B$ is the lift of $B$.

For the second step, we require an approximation argument.  Let $B\subset X\backslash \Sigma$ have center $c\in X\backslash \Sigma$ and radius $t>0$. Let $d = \|B - \Sigma\|$ denote the Hausdorff distance of $B$ from $\Sigma$.   
For sufficiently small $\delta > 0$ (with $\delta\ll d$), let $B_\delta$ and $B_{-\delta}$ denote concentric balls also centered at $c$, with radii $t+\delta$ and $t-\delta$, respectively. Fix $R>0$ and $\epsilon$ such that $\epsilon \ll \delta$.

Let  $\mathcal L(R)$ denote the set of connected components of $\mathcal C(x,R)\cap B$. Similarly, let $\mathcal A_\delta(R)$ and $\mathcal A_{-\delta}(R)$ denote the connected components of $\mathcal A(\widetilde x, R,R-\epsilon)\cap \widetilde{B_{\delta}}$ and $\mathcal A(\widetilde x, R,R-\epsilon)\cap \widetilde{B_{-\delta}}$, respectively (see Figure 6).

Note that to each region $A_{-\delta}\in\mathcal A_{-\delta}(R)$, we can associate a segment $L\in \mathcal L(R)$, namely the segment which corresponds to the boundary component of $A_{-\delta}$ furthest away from the associated singularity. Similarly, for each $L\in\mathcal L(R)$ we can associate a region $A_\delta\in\mathcal A_\delta(R)$ (see Figure 6). Hence we have the following inclusions:

$$\mathcal A_{-\delta}(R)\hookrightarrow \mathcal L(R)\hookrightarrow \mathcal A_{\delta}(R).$$
Note that the reverse inclusions do not necessarily hold.

   \begin{figure}\label{fig:sector}
\centerline{
\begin{tikzpicture}[scale=0.8]
\draw [thick] (0,0) circle [radius=2.5];
\draw [thick, dashed] (0,0) circle [radius=1.5];
 \draw [red,domain=-30:30] plot ({-9+8*cos(\x)}, {8*sin(\x)});
  \draw [red,ultra thick,domain=-16:16] plot ({-9+8*cos(\x)}, {8*sin(\x)});
  \draw [red,thick,dashed,domain=-30:30] plot ({-9+7*cos(\x)}, {7*sin(\x)});
 \node at (2.7, -0.5) {$B$};
  \node at (1.7, 1) {$B_{-\delta}$};
 \node at (-0.9,1.9) {$L$};
    \node at (-1.1, -0.04) {$A_{-\delta}$};
\end{tikzpicture}
\hskip 1cm
\begin{tikzpicture}[scale=0.8]
\draw [thick] (0,0) circle [radius=2.5];
\draw [thick, dashed] (0,0) circle [radius=3.5];
 \draw [red,domain=-23:23] plot ({-9+8*cos(\x)}, {8*sin(\x)});
  \draw [red,ultra thick,domain=-16:16] plot ({-9+8*cos(\x)}, {8*sin(\x)});
  \draw [red,thick,dashed,domain=-23:23] plot ({-9+7*cos(\x)}, {7*sin(\x)});
 \node at (2.7, -0.5) {$B$};
  \node at (3.5, 2) {$B_{\delta}$};
    \node at (-1.85, -2.2) {$A_{\delta}$};
    \node at (-0.8, 1) {$L$};
\end{tikzpicture}
}
\caption{
(a) We can associate to a region $A_{-\delta}\in \mathcal A_{-\delta}(R)$, a segment $L\in\mathcal L(R)$; and 
(b)  We can associate to $L$ the region $A_\delta \in  \mathcal A_{\delta}(R)$ 
}
\end{figure}
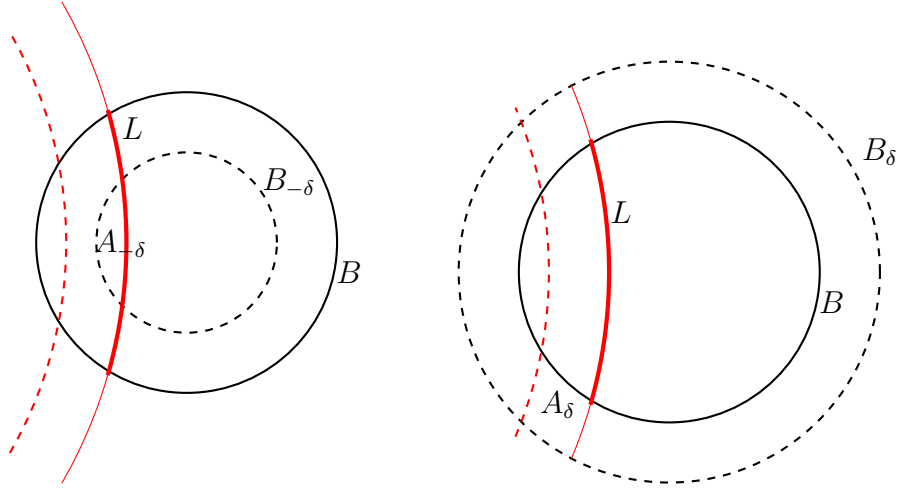

For $L\in\mathcal L(R)$, we will compare $\ell(L)\epsilon$ to the volume $\hbox{Vol}_{\widetilde X}(A_\delta)$ of the associated region $A_\delta\in \mathcal A_\delta(R)$. Using the assumption that $\epsilon \ll \delta$ and 
a little 
Euclidean geometry, it follows that 
$$\ell(L)\epsilon\leq \frac{\hbox{Vol}_{\widetilde X}(A_\delta)}{(1-\frac{\epsilon}{2d})}.$$
Similarly, for $A_{-\delta}\in\mathcal A_{-\delta}(R)$, we can compare $\hbox{Vol}_{\widetilde X}(A_{-\delta})$ to $\ell(L)\epsilon$ for the associated $L\in\mathcal L(R)$ and deduce that 
$$\hbox{Vol}_{\widetilde X}(A_{-\delta})\leq L\epsilon.$$

By summing up the contributions from the aforementioned connected components and using the bounds above, it follows that
$$\hbox{Vol}_{\widetilde X}(\mathcal A(\widetilde x,R,R-\epsilon)\cap \widetilde{B_{-\delta}})\leq  \ell(\mathcal C(x,R)\cap \widetilde B)\epsilon\leq \frac{\hbox{Vol}_{\widetilde X}(\mathcal A(\widetilde x,R,R-\epsilon)\cap \widetilde{B_{\delta}})}{\left(1-\frac{\epsilon}{2d}\right)}. $$
Using the asymptotic formula  (5.1) for annuli and the above bounds, we can deduce that 
$$
\begin{aligned}
 \frac{C(B_{-\delta})}{C(B)}\frac{(1-e^{-h\epsilon})}{\epsilon}
&\leq 
\liminf_{R \to \infty}
\frac{\ell(\mathcal C(x,R) \cap B) }{
(C(B)/h)e^{hR}}\cr
&\leq \limsup_{R \to \infty}
\frac{\ell(\mathcal C(x,R) \cap B) }{
(C(B)/h)e^{hR}}
\cr
&\leq 
\frac{C(B_{\delta})}{C(B)}\frac{(1-e^{-h\epsilon})}{\epsilon}\frac{1}{\left(1-\frac{\epsilon}{2d}\right)}.
\cr
\end{aligned}
$$
Since $\left ( 1 - e^{-h\epsilon } \right)/\epsilon = h + O(\epsilon)$ independently of $R$,
letting $\epsilon \to 0$
we can deduce that 
$$
 \frac{C(B_{-\delta})}{C(B)}h
\leq 
\liminf_{R \to \infty}
\frac{\ell(\mathcal C(x,R) \cap B) }{
(C(B)/h)e^{hR}}
\leq \limsup_{R \to \infty}
\frac{\ell(\mathcal C(x,R) \cap B) }{
(C(B)/h)e^{hR}}
\leq 
\frac{C(B_{\delta})}{C(B)}h.
$$
We can  deduce an asymptotic formula for $\ell(\mathcal C(x,R)\cap B)$ by letting $\delta\rightarrow 0$ and using the absolute continuity of the measure $\mu(A):=C(A)/C(X)$ to conclude that
$$
\begin{aligned}
\ell(\mathcal C(x,R) \cap B)
&\sim C(B)e^{hR}\hbox{ as } R \to \infty.
\end{aligned}
$$
Finally, we can prove Theorem B by considering the above asymptotic formula, Theorem A and Lemma \ref{lem:coeff}:
$$\lim_{R \to \infty} \mu_R(B)
= \lim_{R \to \infty} \frac{\ell(\mathcal C(x, R)\cap B)}{\ell(\mathcal C(x, R))}=\lim_{R\rightarrow \infty} \frac{C(B)e^{hR}}{(C/h)e^{hR}}=\frac{C(B)}{C(X)}=\mu(B),$$  
for all (small) balls $X\backslash \Sigma$.


\begin{rem}
Although the probability measures $\mu$ and  $\hbox{Vol}_X$ are equivalent they are not equal.
\end{rem}

\section{Distribution of closed geodesics}

To put Theorem A and Theorem B into context, we can compare  these  to corresponding results for closed geodesics on $X$. We first describe how to recover an unpublished result of Eskin and Rafi and then we will present a new distribution result for closed geodesics.




\subsection{Closed geodesics}



The following definition gives a natural characterization of closed geodesics on translation surfaces that are of interest to us.
\begin{definition} 
A closed geodesic on a translation surface is a  saddle connection path corresponding to an  (allowed) finite string of oriented saddle connections  $q = (s_1, \ldots, s_n)$, of length $|q|=n$ and up to cyclic permutation, with the additional requirement that $s_ns_1$ is a saddle connection path. We say that $q$ is primitive if it is not a multiple concatenation of a shorter closed geodesic.
\end{definition}

It is convenient to introduce the following notation.

\medskip
\noindent
{\bf Notation.}
Let $\mathcal Q(T)$ denote the set of oriented  primitive closed geodesics on $X$ of length less than or equal to $T$.\footnote{Formally we will count primitive geodesics, which does not include repeated geodesics, but for the purposes of asymptotic counting there is no difference.}
Let $\mathcal Q:=\bigcup_{T>0}\mathcal Q(T)$ denote the set of all oriented primitive  closed geodesics on $X$.
\bigskip

We want to count the number
$\pi(T) := \#\mathcal Q(T)$
of oriented primitive closed geodesics of length at most $T$.  
We adopt the convention that we do not count closed geodesics that do not pass through a singularity, thus avoiding the complication of having uncountably many closed geodesics of the same length.\footnote{Alternatively, we could count one such geodesic from each family but then their growth would only be polynomial and this would not effect the asymptotic.} 

It is easy to show that the exponential growth rate of $\pi(T)$ is equal to the volume entropy of the surface, i.e.,

$$h=\lim_{T\rightarrow \infty}\frac{1}{T}\log \pi(T).$$

\medskip
\noindent
{\bf Notation.}
Given a saddle connection $s_0 \in \mathcal S$, we define
 $\pi_{s_0}(T):= \sum_{q\in\mathcal Q(T)}  \frac{\ell_{s_0}(q)}{\ell(q)}$, where 
 $\ell_{s_0}(q)$ is the length contribution from the saddle connection $s_0$ to $\ell(q)$ (i.e., if $s_0$ occurs $m$ times in $q=(s_{1}, \ldots, s_{n})$, then $\ell_{s_0}(q)=m\ell(s_0)$).
 

For any $T$ sufficiently large we can associate a probability measure $m_T(A) = \frac{1}{\pi(T)} \sum_{q\in\mathcal Q(T)}\frac{\ell_A(q)}{\ell(q)}$,  where
$A$ is a Borel set and 
 $\ell_A(q)$ denotes the length of the part of $q$ which lies in $A$.

\smallskip


\begin{namedthm*}{Theorem C}\label{closed}
Let $X$ be a translation surface with at least one singularity.
\begin{enumerate}
\item 
Then
 $$\pi(T) \sim \frac{e^{hT}}{hT} \hbox{  as  } T \to \infty,
 \hbox{  i.e., 
$\lim_{T \to \infty}\frac{\pi(T)}{ e^{hT}/hT}=1$.}
 $$
\item
For each $s_0 \in \mathcal S$, there exists $0<v(s_0)<1$ such that
$$
\lim_{T \to \infty} \frac{\pi_{s_0}(T)}{\pi(T)} = v(s_0).
$$
In particular, $v$  gives  a probability vector  on $\mathcal S$.
\item
The measures $m_T$ converge 
in the weak-star topology 
to a probability measure $\nu$ which is 
singular with respect to the volume measure on $X$.
\end{enumerate}
\end{namedthm*}
Part 1 of Theorem C is analogous to Theorem A. Parts 2 and 3 are distribution results for saddle connections and closed geodesics, respectively. We will prove part 1 and 2 and note that 
proof 3 can be deduced from part 2.
\begin{rem}
An alternative formulation of the distribution result in part 2 of Theorem C, would be to average the length contribution from $s_0$ across all of the geodesics in $\mathcal Q(T)$ and then obtain the following limit
$$\lim_{T\rightarrow \infty}\frac{\sum_{q\in\mathcal Q(T)} \ell_{s_0}(q)}{\sum_{q\in\mathcal Q(T)}\ell(q)}=v(s_0)$$
by analogy with \cite{parry2}.

Similarly, we obtain an alternative formulation of part 3 of Theorem C, as follows

$$\lim_{T\rightarrow \infty}\frac{\sum_{q\in\mathcal Q(T)}\ell_A(q)}{\sum_{q\in\mathcal Q(T)} \ell(q)}=\nu(A).$$
\end{rem}

The first part of this theorem was announced by Eskin and Rafi \cite{er}.
This result is of  the same  general form as the well known asymptotic formula for closed geodesics for negatively curved Riemannian surfaces.
Two of the classical  approaches that  are successful for surfaces of  negative curvature (the approaches of Selberg and Margulis)  do not have natural analogues in the present context; however, the method of dynamical zeta functions can be  applied  (see \cite{parry1}) and  bears similarities to the proof of Theorem A.



Part 3 of Theorem C was proved by
the  distribution result by Call, Constantine, Erchenko, Sawyer and Work \cite{call} using a completely different method.


\subsection{Zeta functions}
We now present  the definition of the zeta functions that will be used in the proof of part 1 of Theorem C

\begin{definition}
We can formally define the {\it zeta function}
 by the Euler product
 $$
\zeta(z) = \prod_{q\in\mathcal Q} \left( 1 - e^{- z \ell(q)}\right)^{-1},  z \in \mathbb C
$$
where the product is over all oriented primitive closed geodesics.
\end{definition}

This converges to a non-zero analytic function for $Re(z)>h = h(X)$.

We next  give   the definition of a modified zeta function  that will be used in the proof of 
part 2 of 
Theorem C.

\begin{definition}
We can formally define a {\it modified zeta function} for a given  $s_0\in \mathcal S$
 by
 $$
\zeta_{s_0}(z,t) = \prod_{q\in \mathcal Q} \left( 1 - e^{- z \ell(q) + t \ell_{s_0}(q)}\right)^{-1},  z \in \mathbb C \hbox{ and } t \in \mathbb R.
$$
\end{definition}

Given $t\in\mathbb R$, this 
converges  to a non-zero  analytic function for $Re(z)$ sufficiently large. Clearly, when $t=0$, $\zeta_{s_0}(z,t)=\zeta(z)$.

The proof of Theorem C requires us to work  with a  different  presentation of these zeta functions. 
Let 
$\mathcal S_n$
denote the set of oriented saddle connection strings
$p = (s_1, \ldots, s_n)$
 of  length $n$ corresponding to general oriented
(not necessarily primitive)
 closed geodesics $q$. 
Each element $q \in \mathcal Q$ consisting of $n$ saddle connections will give rise to 
 $n$   elements of $\mathcal S_n$ corresponding to the cyclic permutations.
 For $p \in 
\mathcal S_n
$ let $\ell(p):=\sum_{i=1}^n \ell(s_i)$ and, given $s\in \mathcal S$,  we let $\ell_s(p)$ denote the length contribution from $s$ to $\ell(p)$.

\begin{lemma}\label{lem:zetas}
For a given $t\in\mathbb R$, for $Re(z)$ sufficiently large, we can write
$$\zeta_{s_0}(z,t) = \exp \left(   \sum_{n=1}^\infty \frac{1}{n} \sum_{p \in
    \mathcal S_n
    } e^{-z \ell (p)+t\ell_{s_0}(p)}\right).\eqno(6.1)$$
    In particular, when $t=0$ we can write
    $\zeta(z) = \exp \left(   \sum_{n=1}^\infty \frac{1}{n} \sum_{p \in
    \mathcal S_n
    } e^{-z \ell (p)}\right).$
\end{lemma}
\begin{proof}
This is a routine bookkeeping exercise.
We can first write
 $$
 \begin{aligned}
\zeta_{s_0}(z,t) 
&= \exp\left(-\sum_{q\in\mathcal Q} \log(1-e^{-z\ell(q) + t \ell_{s_0}(q) })\right)\cr
&=\exp\left(\sum_{q\in\mathcal Q}\sum_{m=1}^\infty \frac{e^{-zm\ell(q)    + t m\ell_{s_0}(q)    } }{m}\right),
 \end{aligned}
 \eqno(6.2)
$$
using the Taylor expansion for $\log(1-x)$.

Given $k \geq 1$, 
let 
$\mathcal S_k^{prim}\subset \mathcal S_k$ 
denote the set of (allowed) oriented saddle connection strings $p= (s_1,\ldots ,s_k)$ corresponding to oriented primitive closed geodesics $q$ which consist of $k$ saddle connections. 
In particular, each $q$ contributes $k$ strings in  
$\mathcal S_k^{prim}$ (by cyclic permutations). For each $m \geq 1$ we can write 
$$\sum_{q\in \mathcal Q}e^{-zm\ell(q) +tm \ell_{s_0}(q)}
=\sum_{k=1}^\infty\frac{1}{k}\sum_{p\in 
\mathcal S_k^{prim}
}e^{-zm\ell(p)+t m \ell_{s_0}(p)   }.$$
Using the above equation and (6.2)  we see that 
$$
\begin{aligned}
\zeta_{s_0}(z,t)
&=\exp\left(\sum_{k=1}^\infty\sum_{p \in 
\mathcal S_k^{prim}}
\sum_{m=1}^\infty \frac{e^{-zm\ell(p)    + t m\ell_{s_0}(p)    } }{km}\right)\cr
&=\exp\left(\sum_{n=1}^\infty\sum_{p'\in
\mathcal S_n
} 
\frac{e^{-z\ell(p')    + t \ell_{s_0}(p')    }}{n}\right),
\end{aligned}
$$
where we have set $n=km$ and replaced 
$p \in 
\mathcal S_k^{prim}
$ 
and $m \geq 1$ by $p' \in \mathcal S_{n}$.
\end{proof}

The asymptotic results for closed geodesics follow from analytic properties of the above zeta functions (by analogy with the way in which the prime number theorem follows from analytic properties of the Riemann zeta function).

\subsection{Extending the zeta function(s)}
We want to now consider  $z\in \mathbb C$ with $Re(z)>0$ and $t\in\mathbb R$. For part 1 of Theorem C, it suffices to set $t=0$ which leads to some simplifications in the statements below. To extend the zeta functions, it is convenient to introduce the following matrices (generalizing the $M_z$ from Definition 3.5).

\begin{definition}
 Given a choice of saddle connection $s_0 \in \mathcal S$, we  can consider the infinite matrix $K_{z,t}$  with rows and columns  indexed by $\mathcal S$ where 
$$
K_{z,t}(s,s') =
\begin{cases}
 e^{-z \ell(s') + t \ell_{s_0}(s')} & \hbox{ if  $ss'$ is an allowed  geodesic path,} \cr
 0 & \hbox{ otherwise }
 \end{cases}
$$
where the rows and columns are indexed by saddle connections partially ordered by their lengths.
\end{definition}


Since it is easier to deal with finite matrices, 
we can write the matrix $K_{z,t}$ as
$$
K_{z,t} = \left( 
\begin{matrix}
A_{z,t}  & U_{z,t} \cr
V_{z,t} &B_{z,t} \cr 
\end{matrix}
\right),
$$
where $A_{z,t} $ is the $k\times k$ finite sub-matrix of $K_{z,t}$ corresponding to the first $k\in\mathbb N$, say, oriented  saddle connections although  the other sub-matrices $U_{z,t} ,V_{z,t} ,B_{z,t} $ are infinite (cf. \cite{hk}).  



Recall that for the proof of Theorem A, for any $\epsilon>0$, we obtained a meromorphic extension of $\eta(z)$ to the half-plane $Re(z)>\epsilon$, whose poles occur at $z$ for which 1 is an eigenvalue of the matrix $W_{z,t}:=A_{z,t}+U_{z,t}(I-B_{z,t})^{-1}V_{z,t}$. We will pursue a similar strategy here.
To this end,  consider two formally defined auxiliary functions for $s_0\in \mathcal S$,  $z \in \mathbb C$ and $t\in\mathbb R$: 
$$
f_{s_0}(z,t)  = \det (I -  W_{z,t}) 
\hbox{
and
}g_{s_0}(z,t)=\hbox{exp}\left(-\sum_{n=1}^\infty \frac{1}{n}\sum_{
p \in
\mathcal S_n(k)
}
e^{-z\ell(p)+t\ell_{s_0}(p)}\right),$$
where 
$\mathcal S_n(k) \subset \mathcal S_n$
denotes the set of oriented saddle connection strings $p = (s_1, \ldots, s_n)$
  of length $n$ corresponding to
oriented closed geodesics on $X$, and for which 
all of the   $s_j$ ($1 \leq j \leq n$) are  disjoint from the first $k$ saddle connections in the ordering on $\mathcal S$.


\begin{lemma}
Let $s_0 \in \mathcal S$.
Fix $\epsilon>0$ and let $|t|<\epsilon$. Provided
$k$ (i.e., the size of $W_{z,t}$)  is sufficiently large,
 the functions $g_{s_0}(z,t)$ and $f_{s_0}(z,t)$ are analytic on $Re(z)>\epsilon$.
 \end{lemma}
 \begin{proof}
Let $\mathcal S(k) \subset \mathcal S$  consist of  those  saddle connections $s$ 
which are {\it not}  in the first 
 $k$ in the partial ordering. Then
 $$
 \begin{aligned}
\left\vert\sum_{n=1}^\infty \frac{1}{n}\sum_{
p \in 
\mathcal S_n(k)
}e^{-z\ell(p)+t\ell_{s_0}(p)}\right\vert  &\leq \sum_{n=1}^\infty \left(\sum_{s \in\mathcal S(k)}\left\vert e^{-z \ell (s) +t  \ell_{s_0}( s) }\right\vert \right)^n\cr
 &\leq\sum_{n=1}^\infty \left(\sum_{s\in\mathcal S(k)} e^{-(\epsilon-|t|)\ell(s)  }\right)^{n}.
  \end{aligned}
  $$
  Consequently, $g_{s_0}(z,t)$ is analytic for $Re(z)>\epsilon$ if $\sum_{s\in\mathcal S(k)} e^{-(\epsilon-|t|)\ell(s)}<1$, which holds for $k$ sufficiently large by virtue of the polynomial growth of lengths of saddle connections.
  
 For $f_{s_0}(z,t)$ to be analytic on $Re(z)>\epsilon$, it suffices to show that $(I-B_{z,t})$ is invertible, which holds provided  $||B_{z,t}||<1$. To this end, we note that
 
 $$||B_{z,t}||\leq \sum_{s\in\mathcal S(k)} e^{-(\epsilon-|t|)\ell(s)}$$
 and hence again $||B_{z,t}||<1$ for $k$ sufficiently large.
\end{proof}

We can now use the auxiliary functions $f_{s_0}$ and $g_{s_0}$  to provide an extension of the zeta functions.

\begin{lemma}\label{zetaext} Let $s_0 \in \mathcal S$.
Fix $\epsilon > 0$ and $|t|<\epsilon$. Then $\zeta_{s_0}(z,t)$ has a meromorphic extension of the form $$\frac{1}{ {g_{s_0}(z,t) f_{s_0}(z,t)}}$$ on $Re(z) > \epsilon$.
\end{lemma}

\begin{proof}
By equation (6.1), for $Re(z)>h$ and $|t|$ sufficiently small, we can rewrite $\zeta_{s_0}(z,t)$ in terms of $K_{z,t}$ as follows
$$\zeta_{s_0}(z,t)=\exp\left(\sum_{n=1}^\infty \frac{\hbox{tr}(K_{z,t}^n)}{n}\right),$$
where given a (countable) matrix $A$ we define the formal sum $\hbox{tr(A)}:=\sum_{i=1}^\infty A(i,i)$. Similarly, on $Re(z)>h$ we can formally write
$$g_{s_0}(z,t)=\exp\left(-\sum_{n=1}^\infty \frac{\hbox{tr}(B_{z,t}^n)}{n}\right).$$
Next, for $Re(z)>h$  and $|t|$ sufficiently small, we can write
$$f_{s_0}(z,t)=\det(I-W_{z,t})=\exp\left(-\sum_{n=1}^\infty \frac{\hbox{tr}(W_{z,t}^n)}{n}\right).$$

We claim that $\hbox{tr}(K_{z,t}^n)=\hbox{tr}(B_{z,t}^n)+\hbox{tr}(W_{z,t}^n)$. To see this, first note that 
$\hbox{tr}(B_{z,t}^n)$ is the sum of exponentially weighted oriented edge strings in  $\mathcal E_n(k)$ and 
$\hbox{tr}(W_{z,t}^n)$ is the sum of exponentially weighted oriented edge strings with at least one edge in  
the first $k$ saddle connections. 

By combining  the above observations, it follows that for  $Re(z)>h$ and $|t|$ sufficiently small,
we can write
$$
\begin{aligned}
{\zeta_{s_0}(z,t)}&=\exp\left(\sum_{n=1}^\infty \frac{\hbox{tr}(K_{z,t}^n)}{n}\right)\cr
&=\exp\left(\sum_{n=1}^\infty \frac{\hbox{tr}(B_{z,t}^n)}{n}\right)\exp\left(\sum_{n=1}^\infty \frac{\hbox{tr}(W_{z,t}^n)}{n}\right)\cr
&=\frac{1}{g_{s_0}(z,t) f_{s_0}(z,t)}.
\end{aligned}
$$

By the previous Lemma, $f_{s_0}(z,t)$ and $g_{s_0}(z,t)$ are analytic on $Re(z)>\epsilon$, with  $|t| < \epsilon$ and for $k$ sufficiently large, hence the result follows.
\end{proof}

Fix $\epsilon<h$, $|t|<\epsilon$ and let $k$ be sufficiently large so that $1/\zeta_{s_0}(z,t)=g_{s_0}(z,t)f_{s_0}(z,t)$ is analytic on $Re(z)>\epsilon$. To proceed, we need to understand the location of the poles of the extension of $\zeta_{s_0}(z,t)$ on $Re(z)>\epsilon$. Note that $g_{s_0}(z,t)$ is non-zero and hence poles of the extension of $\zeta_{s_0}(z,t)$ correspond to the zeros of $f_{s_0}(z,t)$ in $Re(z) > \epsilon$, i.e. the values of $z$ such that 1 is an eigenvalue of $W_{z,t}$.

The next lemma states that properties analogous to those  required of $\eta(z)$ in the proof of Theorem A also hold for the zeta functions.

\begin{lemma}\label{lem:zetaproperties}
\begin{enumerate}
\item
The meromorphic extension of  $\zeta(z)$ is analytic for $Re(z) > h$, with a simple pole at $z=h$ which has positive residue, and there are no other poles on the line $Re(z)=h$.
\item 
Fix $s_0 \in \mathcal S$.
Providing $\epsilon > 0$ is sufficiently small and  $t \in (-\epsilon, \epsilon)$,
the meromorphic extension of $\zeta_{s_0}(z,t)$ has  a simple pole at the real value $z=p_{s_0}(t)>0$ where $p_{s_0}:(-\epsilon, \epsilon) \to \mathbb R$ is analytic with $p_{s_0}(0)=h$ and $p_{s_0}'(0) \neq 0$.  
Furthermore, the extension   of $\zeta_{s_0}(z,t)$  is analytic on $Re(z)>p_{s_0}(t)$ and there are no other poles on the line $Re(z)=p_{s_0}(t)$.
\end{enumerate}
\end{lemma}

\begin{proof}
For part 1,  the growth rate of $\pi(T)$ being equal to $h$ implies that $\zeta(z)$ is analytic for $Re(z) > h$. The other properties follow from the proof of Lemma \ref{thmAextn}.

For part 2,  by analytic perturbation theory (see \cite{kato}),
 the matrix 
  $W_{z,t} $ has an eigenvalue $\lambda(z,t)$ with a 
   bi-analytic dependence,  
   for $z$ close to $h$ and $|t|$ sufficiently small, such that $\lambda(h,0)=1$.
   Furthermore, for $\sigma>0$,  $\frac{\partial \lambda(\sigma,0)}{\partial \sigma}|_{\sigma=h} = u A v< 0$ where 
$u, v >0$ are the normalized left and right eigenvectors, respectively,  of $W_{z,0}$ for the eigenvalue $1$, and $A(i,j) = \frac{\partial W_{\sigma}(i,j)}{\partial \sigma}|_{\sigma=h} < 0$ (cf. proof of
Lemma \ref{thmAextn},  or 
 Lemma 4.3 of \cite{cp}).
Since  by  Lemma \ref{zetaext}   the poles of $\zeta_{s_0}(z,t)$ occur where  the matrix  $W_{z,t} $ has $1$ as an eigenvalue, we can apply the Implicit Function Theorem to $\lambda(p_{s_0}(t),t) = 1$ to find an analytic solution $p_{s_0}(t)$.   The Implicit Function Theorem also gives
$$p_{s_0}'(0) = - \frac{\partial \lambda(h,t)}{\partial t}|_{t=0}
/ \frac{\partial \lambda(\sigma,0)}{\partial \sigma}|_{\sigma=h}  > 0$$
since $\frac{\partial \lambda(h,0)}{\partial t}|_{t=0} = u B v < 0$ where 
$B(i,j) = \frac{\partial W_{h,t}(i,j)}{\partial t}|_{t=0} < 0$ (cf. proof of
Lemma \ref{thmAextn} or 
 Lemma 4.3 of \cite{cp}).
 The final part of lemma follows by a similar argument to 
 the proof of Lemma \ref{thmAextn}.


\end{proof}

\subsection{Proof of  part 1 of Theorem C}
Having established the properties of the complex function $\zeta(z)$, the derivation of the asymptotic formula
in part 1 of Theorem C
 follows a classical route (cf. \cite{parry2}, after some trivial corrections).
Using (6.2) with $t=0$ we can write 
$$
- \frac{\zeta'(z)}{\zeta(z)} = \sum_{n=1}^\infty  \sum_{q\in\mathcal Q} 
\ell(q)
e^{-z n\ell(q)}
= \int_0^\infty  e^{-zT} dF(T)\eqno(6.3)
$$
where
$$
F(T)  := \sum_{n \ell(q) \leq T} \ell(q) = 
\sum_{ n \ell(q)  \leq T} \ell(q)  \left[
\frac{T}{\ell(q)}
\right] \leq \pi(T) T,  \eqno(6.4)
$$
with the summation over pairs $(n,q) \in \mathbb N \times \mathcal Q$ provided $n \ell(q)  \leq T$, 
 and $\pi(T) = 
 \hbox{\rm Card}
\{q \in \mathcal Q \hbox{ : }  \ell(q) \leq t\}$.

By part 1 of Lemma \ref{lem:zetaproperties}
we can write $\zeta(z) = \psi(z)/(z-h)$ where  $\psi(z)$  is analytic in a neighbourhood of $Re(z) \geq h$ and non-zero at $h$. Thus 
$$
 \frac{\zeta'(z)}{\zeta(z)} =  \frac{-1}{z-h} +\frac{\psi'(z)}{\psi(z)}. \eqno(6.5)
$$
Comparing (6.3) and (6.5) we can apply  Theorem \ref{tauberian} to deduce that 
$F(T) \sim e^{hT}/h$ as $T \to \infty$. Using  (6.4), it follows that  $\liminf_{T\to \infty} \frac{\pi(T)}{e^{hT}/hT} \geq 1$.

For any $\sigma>h$ and sufficiently large $T > 0$ we can sum the geometric series in  (6.3) to bound
$$
-\frac{\zeta'(\sigma)}{\zeta(\sigma)}  \geq \sum_{\ell(q)  \leq T} \frac{1}{e^{\sigma\ell(q)}}\frac{\ell(q)}{\left(1-e^{-\sigma \ell(q)}\right)}
\geq 
 \sum_{\ell(q) \leq T} \frac{\ell(q)}{\sigma \ell(q)} 
 \frac{1}{e^{\sigma T}}
 =  \frac{1}{\sigma} \frac{\pi(T)}{e^{\sigma T}}.
$$
Thus for any $\sigma' > \sigma$ we have 

$$\frac{\pi(T)}{e^{\sigma'T}}\leq e^{(\sigma-\sigma')T}\sigma\left(-\frac{\zeta'(\sigma)}{\zeta(\sigma)}\right)\rightarrow 0\hbox{ as }x\rightarrow \infty.$$ 
Since $\sigma >h$ is arbitrary, we deduce that 
$ \pi(T)/e^{\sigma T} \to 0$ as $T \to \infty$.      Given $T$ sufficiently large, we choose $y< T$ such that 
$e^{\sigma y} = e^{hT}/T$ and write 
$$
\pi(T) - \pi(y) =  \sum_{y < \ell(q)  \leq T} 1 \leq  \sum_{\ell(q)  \leq T} \frac{\ell(q)}{y}
\leq \frac{F(T)}{y}.
$$
In particular,  by rearranging this inequality we can write 
$$
\pi(T) \frac{hT}{e^{h T}} \leq  \frac{h\pi(y)}{e^{\sigma y }} 
+ \frac{h F(T) }{e^{\sigma y} y} 
=  \frac{h\pi(y)}{e^{\sigma y }} 
+ \frac{h F(T) }{e^{hT} } \left( \frac{1}{\frac{h}{\sigma}- \frac{\log T}{\sigma T}}\right)
$$
so that $\limsup_{T \to \infty} \pi(T) \frac{hT}{e^{h T}}  \leq \frac{\sigma}{h}$.  Since $\sigma > h$ can be chosen arbitrarily, we deduce that $\pi(T) \sim \frac{e^{hT}}{hT}$, as required.


\subsection{Proof of part 2 of Theorem C}

The proof of part 2 is analogous to the proof of part 1 (compare with \S7 of \cite{parry2}). However, one difference is that we differentiate the appropriate zeta function with respect to the second variable $t$:

$$
- \frac{\partial}{\partial t} \log \zeta_{s_0}(z,t)
= \sum_{n=1}^\infty  \sum_{q\in\mathcal Q} 
\ell_{s_0}(q)
e^{-z n\ell(q)}
= \int_0^\infty  e^{-zT} dF_{s_0}(T)\eqno(6.7)
$$
where
$$
F_{s_0}(T) := \sum_{ \ell(q) n \leq T}\ell_{s_0} (q) \leq \pi_{s_0}(T) T \eqno(6.8)
$$
for $T>0$ and
$\pi_{s_0}(T) =  \sum_{ q\in \mathcal Q(T)} \frac{\ell_{s_0} (q)}{\ell(q)}$.

By part 2 of  Lemma \ref{lem:zetaproperties}, we can write $\zeta_{s_0}(z,t) = \psi_{s_0}(z,t)/(z-p_{s_0}(t))$ where  $\psi_{s_0}(z, t)$  is analytic in a neighbourhood of $Re(z) \geq p_{s_0}(t)$ and thus 
$$
 \frac{\partial}{\partial t} \log \zeta_{s_0}(z,t)|_{t=0}=  \frac{-p_{s_0}'(0)}{z-h} +\frac{\psi_{s_0}'(z,0)}{\psi_{s_0}(z,0)},\eqno(6.9)
$$
using the fact that $p_{s_0}(0)=h$.  We can write $v(s_0) = p_{s_0}'(0)$.

Comparing (6.7) and (6.9) we can apply 
Theorem \ref{tauberian} to deduce that 
$F_{s_0}(T) \sim v(s_0) e^{hT}$ as $T \to \infty$ and thus using (6.8) we have $\liminf_{T\to \infty} \frac{\pi_{s_0}(T)}{e^{hT}/hT} \geq v(s_0)$.

For any $\sigma>h$,  a similar argument to that in the previous subsection gives
$$
- \frac{\partial}{\partial t} \log \zeta_{s_0}(\sigma,t)|_{t=0}  \geq \sum_{\ell(q) \leq T}\frac{1}{\sigma}\frac{\ell_{s_0}(q)}{\ell(q)}\frac{1}{e^{\sigma T}}
\geq 
\frac{1}{\sigma} \frac{\pi_{s_0}(T)}{e^{\sigma T}}.
$$
Since  $\sigma> h$ is arbitrary, $\pi_{s_0}(T)/e^{\sigma T}\rightarrow 0$ as $T\rightarrow \infty$. Again, as in the previous subsection, we choose $y<T$ such that $e^{\sigma y}=e^{hT}{T}$ and write

$$
\pi_{s_0}(T) - \pi_{s_0}(y) =  \sum_{y < \ell(q)  \leq T} \frac{\ell_{s_0}(q)}{\ell(q)}
\leq \frac{F_{s_0}(T)}{y}
$$
and thus as before 
$$
\pi_{s_0}(T) \frac{hT}{e^{h T}} \leq 
 \frac{h\pi_{s_0}(y)}{e^{\sigma y }} 
+ \frac{h F_{s_0}(T) }{e^{hT} } \left( \frac{1}{\frac{h}{\sigma}- \frac{\log T}{\sigma T}}\right)
$$
so that $\limsup_{T \to +\infty} \pi_{s_0}(T) \frac{hT}{e^{h T}}  \leq \frac{\sigma}{h} v(s_0)$.  Since $\sigma > h$ can be chosen arbitrarily we deduce that $\pi_{s_0}(T) \sim \frac{e^{hT}}{hT} v(s_0)$, as required.

Finally, we obtain part 2 of Theorem C as follows
$$\lim_{T\rightarrow \infty}\frac{\pi_{s_0}(T)}{\pi(T)}=\lim_{T\rightarrow \infty} \frac{(e^{hT}/hT)v(s_0)}{(e^{hT}/hT)}=v(s_0).$$

Part 3 of  Theorem C can be deduced from part 2. The measure $\nu$ obtained is singular with respect to the volume measure $\hbox{\rm (Vol)}_X$
(and thus with respect to $\mu$)  since one can show that the Borel set $Y= \cup_{s \in \mathcal S}s$ corresponding to the union of the saddle connections  has 
$\nu\left( Y \right) = 1$,
but $\hbox{\rm (Vol)}_X\left(Y \right) = 0$.

\section{Final comments and questions}

\begin{enumerate}
\item
Stronger asymptotic results might  involve error terms. 
For example,  if   there exist closed geodesics $q,q'\in\mathcal Q$ such that the ratio 
$\alpha = \frac{\ell(q)}{\ell(q')}$ is diophantine (i.e., there exists $\tau > 2$ such that 
$|\alpha - \frac{a}{b}| \geq \frac{1}{b^\tau}$ has only finitely many rational solutions $\frac{a}{b}$)
then  one could  show  that there  exists $\beta > 0$
such that $\pi(T) =  (e^{hT}/hT) \left(1+ O(T^{-\beta}) \right)$ as $T \to \infty$ (compare with \cite{ps}).  On the other hand, the error term  can never  be improved to an exponential error term, i.e., there is no $\delta > 0$ such that 
$\pi(T) =  \hbox{Li}(e^{hT})\left(1+ O(e^{-\delta T}) \right)$ as $T \to \infty$ 
(where $\hbox{Li}(T) = \int_2^T (\log u)^{-1} du$)
since this would necessarily  require  $\zeta(z)$ being non-zero and analytic on 
the domain $Re(z) > h -\delta$, except 
for a simple pole at $z=h$, which is incompatible with the matrix approach to the extension.
\item
Other potential strengthenings of the basic distribution theorem (Theorem C.3) might include, for example, a large deviation result \cite{kifer}.  
\item It should be straightforward  to modify our approach so as to weight the geodesics using the exponential of the integral along the geodesic of a 
 suitable function.  In this case the entropy would be replaced by a pressure function and the asymptotic counting function and distribution result would be replaced by correspondingly weighted versions (cf. \cite{parry3}).
 
\end{enumerate}

\Addresses

\end{document}